\documentclass[a4paper,12pt, leqno ]{article}
\usepackage{latexsym}
\usepackage{amsmath}
\usepackage{amssymb}
\usepackage{enumerate}
\usepackage{pict2e} 
\usepackage{tikz}
\usetikzlibrary{trees}
\usetikzlibrary{calc}
\usetikzlibrary{intersections}
\usepackage{theorem}

\usepackage{epic,eepic}
\usepackage{color}
\usepackage{float}

\newtheorem{theorem}{Theorem}[section]
\newtheorem{proposition}[theorem]{Proposition}
\newtheorem{lemma}[theorem]{Lemma}
\newtheorem{corollary}[theorem]{Corollary}

\theorembodyfont{\rmfamily}
\newtheorem{proof}{\textmd{\textit{Proof.}}}

\newtheorem{remark}[theorem]{Remark}

\newtheorem{definition}[theorem]{Definition}

\makeatletter

\@addtoreset{equation}{section}
\makeatother

\newcommand{\qedd}{\hfill \Box}
\newcommand{\ve}{\varepsilon}

\newcommand{\R}{\ensuremath{\mathbb{R}}}

\newcommand{\E}{\ensuremath{\mathbb{E}}}
\newcommand{\Sph}{\ensuremath{\mathbb{S}}}
\newcommand{\C}{\ensuremath{\mathbb{C}}}

\def\Ric{\mathop{\mathrm{Ric}}\nolimits}

\title{Geodesics on strong Kropina manifolds
	\footnote{
		Mathematics Subject Classification (2010)\,:\,53C60, 53C22.}
}

\author{By  Sorin V. Sabau, Kazuhiro Shibuya and Ryozo Yoshikawa }
\date{}
\begin{document}

\maketitle



\begin{abstract}
We study the behavior of the geodesics of strong Kropina spaces. The global and local aspects of geodesics theory are discussed. We illustrate our theory with several examples.
\end{abstract}

\section{Introduction}
\hspace{0.2in}  The study of local and global behavior of geodesics is one of the main topics in  differential geometry.  The characteristic features of the Riemannian manifolds allow a quite detailed study of this topic, and the literature is exhaustive (see  \cite{Ber}, \cite{B}, \cite{Ca}, \cite{Kl} and many other places).

On the other hand, in Finsler geometry, due to the computational difficulties and other peculiarities, the results on the geometry of geodesics are by far less well studied than in the Riemannian setting.
In particular,  the literature concerning the global behavior of long geodesics is limited (see \cite{BCS}, \cite{Sh}).

It is known that a special class of Finsler spaces, namely the Randers spaces, appears as the solution of Zermelo's navigation problem. In particular, when the  breeze can be modeled as a Killing vector field of the underlying Riemannian structure, one can describe in detail the behavior of both short and long geodesics, conjugate locus, cut locus, etc. (see \cite{BCS}, \cite{R}).

Recently we have shown that Kropina spaces, despite of the fact of being conic Finsler structures, can be also obtained as solutions of the Zermelo's navigation problem. See \cite{YS1} for a global theory of Kropina metrics followed by a complete classification of the constant flag curvature Kropina spaces. Indeed, a Kropina space is of constant flag curvature if and only if its Zermelo's navigation data is a unit length Killing vector field $W$ on a constant sectional curvature Riemannian manifold $(M,h)$. Remarkably, unlike the Randers case, only odd dimensional spheres and the Euclidean space admit Kropina metrics of constant flag curvature
(\cite{YS1}).

In the present paper we study the geometry of geodesics of  a Kropina space with the fundamental  metric obtained by the Zermelo's navigation problem with navigation data given by  a unit Killing vector field $W$ on a Riemannian manifold $(M,h)$.
 We call such a space {\it a strong Kropina space}. The Riemannian structure $(M,h)$ does not need to be of constant sectional curvature. We point out that strong Kropina spaces are conic Finsler spaces (see \cite{JS}, \cite{JS2}, \cite{JV} and references within for general theory) and therefore the theory of classical Finsler spaces does not apply.
Indeed, geodesics in a Kropina  space are $F$-admissible curves whose tangent vector is parallel along the curve.
In other words, they are solutions of a certain second order system of differential equations with initial conditions $c(0)=p\in M$ and $\dot{c}(0)\in A_p$, where $A_p\subset T_pM$ is the conic domain of the Kropina metric $F=\alpha^2/\beta$, i.e. the region of $T_pM$ where $F$ is well defined as Finsler norm. We would like to study the geometrical of this special kind of geodesics. We are in special interested in the local and global behavior of these conic geodesics as well as cut locus related properties. 

We point out that since Kropina metrics are special conic metrics, most of their geometrical  properties can be deduced from the general case of conic metrics. Even though this fact makes our paper mainly a review paper, we hope this paper will be a useful one for those learning about other Finsler metrics slightly different from the classical one, in special Kropina type metrics. 

\bigskip

The paper is organized as follows. 

In \S 2 we discuss the general theory of Kropina metrics. In \S 2.1  we recall the construction of Kropina metrics as singular solution of Zermelo's navigation
problem (see \cite{YS1}). We point out that unlike the Randers case, the correspondence between the $(\alpha,\beta)$-data and the navigation data $(h,W)$ is not one-to-one. Since our characterization of globally defined Kropina metrics is done by means of a globally defined unit vector field on the manifold $M$, there are obvious topological restrictions to the existence of such manifolds (see \S 2.2). We give here a complete list of differentiable manifolds that admit non-vanishing vector fields.

In \S 3 we restrict ourselves to a special family of global Kropina manifolds, namely the strong Kropina manifolds, obtained as solutions of Zermelo's navigation problem for the navigation data $(h ,W)$, where $h$ is a Riemannian metric and $W$ a unit Killing field, with respect to $h$, on the base manifold $M$. Clearly, to the topological restrictions on the manifold $M$ presented in \S 2.2, there are curvature restrictions on the Riemannian metric $M$ for globally strong Kropina metrics to exist. The notion of quasi-regular unit Killing fields, that is unit Killing fields all of whose integral curves are closed,  appears here. These Killing fields will play an important role in the following sections. A topological classification of strong Kropina manifolds are obtained in some special cases (see Theorems \ref{thm: topological classification in two dimensions} and \ref{thm: topological classification in the complete case}).

In \S 4 we start by showing that the Kropina geodesics are obtained in a similar manner with the Randers geodesics (\cite{R}), namely by the deviation of a Riemannian geodesic by the flow of $W$ (see \S 4.1).  We define the exponential map of a conic map in \S 4.2 and study its basic properties. Finally, we consider the separation of two points $p$ and $q$ on the base manifold $M$, notion similar to the Finsler distance. Due to the conic character of the metric, obviously  there are differences with the classical case (see \S 4.3). Using the separation function, the geodesic balls in a strong Kropina manifold can be defined, and we show that, even though we are working with a conic metric, these balls still give the topology of the base manifold.  

The minimality of short geodesics and a conic version of Gauss Lemma are presented in \S 5.1, and minimality of long geodesics is discussed in \S 5.2. We explain here that a Hopf-Rinow type theorem for generic Kropina manifolds is not possible.  

One characteristic feature of conic metrics is that they are not geodesically connected, that is even assuming forward completeness of the metric, for arbitrary chosen two points $p$, $q$ the existence 
of a geodesic segment joining these points cannot be proved. Indeed, we present in \S 6.1 the example of the Euclidean space endowed with a strong Kropina metric $F$, forward complete, but having points that cannot be joined by $F$-geodesic segments. However, in the case when the unit Killing field is quasi-regular we prove that $M$ is geodesically connected and hence  a Hopf-Rinow type theorem can be now formulated under this assumption (see \S 6.2). In the following we present the example of strong Kropina metrics constructed on Riemannian homogeneous spaces (\S 6.3).

We study the conjugate and cut points  along the $F$-geodesics in \S 7. 

In the final section \S 8 of the paper we consider some representative examples for the study of the global behavior of $F$-geodesics. Besides the Euclidean case, we study here the geodesics together with their cut locus for sphere (example of positively curved manifold),  cylinder and flat torus (locally Euclidean manifolds). We end the paper with an appendix that contains some local computations needed to determine the geodesic equation and some of their properties.

{\bf Acknowledgements.} We express our gratitude to M. Javaloyes, N. Innami, H. Shimada and K. Okubo for many useful discussions. We indebted to the aonymous referee whose suggestions and insights have improved the exposition of the paper considerably. We are also grateful to U. Somboon for her help with some of the drawings.


\section{Globally defined Kropina metrics}
\subsection{Singular solutions to the Zermelo's navigation problem}\label{sec: Zermelo Navigation}

\hspace{0.2in}  Let $(M,a)$
be an $n$-dimensional Riemannian manifold $M$, let $\alpha(x,y)=|y|_a$ be the $a$-length of the tangent vector $y=y^i\frac{\partial}{\partial x^i}\in T_xM$
 and  $\beta$ 
 be a 1-form on $M$ which we naturally identify with the linear 1-form $\beta=b_i(x)y^i$ on the tangent bundle $TM$. We denote hereafter the coordinates on $M$ by $x=(x^1,\dots,x^n)$ and the natural coordinates on $TM$ by $(x,y)=(x^1,\dots,x^n;y^1,\dots,y^n)$.

We recall that the Finsler metric 
\begin{equation}\label{Kropina metric}
F(x,y):=\frac{\alpha^2(x,y)}{\beta(x,y)} 
\end{equation}
is called a {\it Kropina metric} and it was introduced by V. K. Kropina in 1961 (see \cite{K} and \cite{AIM}). This is not a classical Finsler metric, but a {\it conic} one, that is  $F:A\to \R^+$ is a Finsler metric defined on the conic domain $A\subset TM$ (see \cite{YS1}).  Here $A=\cup_{x\in M}A_x$ and $A_x:=A\cap T_xM$ is given by
\begin{equation}\label{2.2}
A_x=\{y\in T_xM\ : \ \beta(x,y)>0\}.
\end{equation}
Hence, in order to emphasize the domain of  definition $A$ we denote a Kropina space by $(M, F=\alpha^2/\beta : A\to \R^+)$.

We have recently shown \cite{YS1} that Kropina metrics are singular solutions of the  Zermelo's navigation problem. Indeed, in 1931, E. Zermelo studied the following problem:

{\it Suppose a ship sails the sea and a wind comes up. How must the captain steer the ship in order to reach a given destination in the shortest time?
}

The problem was solved by Zermelo himself (\cite{Z}) for the Euclidean flat plane and by Z. Shen \cite{Sh2} in a more general setting. This approach turns out to be extremely fruitful in the study of Randers metrics of constant flag curvature  (\cite{BRS}) and the geodesics' behavior in Randers spaces \cite{R}. Its success lies in the fact that Randers metrics $F=\alpha+\beta$, where $\alpha$ and $\beta$ have the same meanings as above, are solutions of the Zermelo's navigation problem in the case when the ``underlying sea" is a Riemannian manifold $(M,h)$ and the ``wind" is a time independent vector field $W$ on $M$ such that $|W|_h<1$. In this case, the travel-time minimizing paths are exactly the geodesics of an associated Randers metric. We point out that the condition  $|W|_h<1$ is essential for obtaining a positive  definite  Randers metric out of the Zermelo's navigation problem.

More precisely, solving the Zermelo's navigation problem is equivalent to finding a norm $F$ on $T_xM$, or a subset of $T_xM$, such that $F(v)=F(u+W)=1$, where
$u$ the velocity of the ship in the absence of the wind, $|u|_h=1$, and hence $v=u+W$ is the ship velocity under windy conditions.

We have considered in \cite{YS1}
a special case of Zermelo's navigation problem, namely the case when the wind is $h$-unitary, i.e. $|W|_h=1$.
\begin{remark}\label{rem: equiv}
An elementary computation shows that
\[
F(x,y)=1\quad \textrm{ if and only if } \quad |y-W|_h=1,
\]
for any $(x,y)\in TM$ such that $h(y,W)>0$.
\end{remark}

If we start with a Riemannian manifold $(M,h)$ and a vector field $W=W^i\frac{\partial}{\partial x^i}$ on $M$ of $h$-unit length, then the solution of
\begin{equation*}\label{eq for F}
\bigg|\frac{y}{F(x,y)}-W  \bigg|_h=1
\end{equation*}
is a function of the form
\begin{equation*}
F(x,y)=\frac{|y|_h^2}{2h(y,W)}.
\end{equation*}

It can be seen that,
due to the condition $|W|_h=1$ everywhere, that is $W$ is nowhere vanishing on $M$, the norm $F$ is a Kropina metric defined on the conic domain
\begin{equation}\label{conic domain by inequality}
A=\{(x,y)\in TM\ : \ h(y,W)>0\}.
\end{equation}

Indeed, by putting
\begin{equation}\label{a_ij, b_i def}
a(x):=e^{-\kappa(x)}h(x),\qquad \beta:=2e^{-\kappa(x)}W^\flat,
\end{equation}
we obtain the metric
\begin{equation}\label{Kropina metric def}
F(x,y)=\frac{|y|_h^2}{2h(y,W)}=\frac{\alpha^2(x,y)}{\beta(x,y)},
\end{equation}
where $\kappa(x)$ is an arbitrary 
smooth function on $M$, and $W^\flat:=(h_{ij}W^j)dx^i$ is the 1-form on $M$ obtained by the Legendre transformation of $W$ with respect to the Riemannian metric $h$, or, equivalently the linear 1-form $W^\flat=h(y,W)$ on $TM$. One can remark that with these notations we have $b^2(x):=|b(x)|_a^2=4e^{-\kappa(x)}$, where $b(x)=(b_i(x))$.

Conversely, if we start with the Kropina metric \eqref{Kropina metric}, defined on the conic domain \eqref{2.2},  then the Riemannian metric $a$ and the 1-form $\beta$ induce the {\it navigation data} $(h,W)$, made of the Riemannian metric $h$ on $M$ and the $h$-unit vector field $W$ given by
\begin{equation}\label{2.5}
h(x):=e^{\kappa(x)}a(x),\quad \textrm{ and }\quad W:=\frac{1}{2}e^{\kappa(x)}\beta^\sharp,
\end{equation}
respectively,
where the vector field $\beta^\sharp$ is the inverse Legendre transform of the 1-form $\beta$  with respect to the Riemannian metric $a$, and
$\kappa(x)=\log\frac{4}{b^2(x)}$.
Obviously, we obtain the initial Zermelo's navigation problem in terms of $h$ and $W$ whose solution is precisely the Kropina metric \eqref{Kropina metric def}.


\bigskip 

\bigskip


\begin{figure}[h]
	\begin{center}
		\setlength{\unitlength}{1cm}
		\begin{tikzpicture}

		\put(-2,14.25){\vector(1,0){5}}
		\put(0.1,11.75){\vector(0,1){5}}
		\draw (3.5,2.1) circle (1cm);
		\draw (2.5,2.1) circle (1cm);
		\put(0.1,14.25){\vector(1.5,1){1.4}}
		\put(0.1,14.25){\vector(-1,1){0.75}}
		\put(0.1,14.25){\vector(1,0){}}
		\put(2.9,1.9){\begin{tiny}
			$W_{p}$
			\end{tiny}}
		\put(3.6,2.7){\begin{tiny}
			$y$
			\end{tiny}}
		\put(1.7,2.3){\begin{tiny}
			$u$
			\end{tiny}}
		\put(3.2,3.5){\begin{scriptsize}
			$ F(y)=1$
			\end{scriptsize}}
		\put(0.7,3.5){\begin{scriptsize}
			$ \Vert u \Vert_{h}=1$
			\end{scriptsize}}
		\put(2.3,-1){\begin{scriptsize}
			$ (a)$
			\end{scriptsize}}

		\put(4.25,14.25){\vector(1,0){5}}
		\put(6.1,11.75){\vector(0,1){5}}
		\put(9,-0.41){\line(0,1){5}}
		\put(9.1,3.5){\begin{scriptsize}
			$ F(y)=1$
			\end{scriptsize}}
		\put(6.1,14.25){\vector(1.5,1){0.85}}
		\put(6.1,14.25){\vector(0.7,2){0.5}}
		\put(9.1,2.75){\begin{tiny}
			$y$
			\end{tiny}}
		\put(8.6,3.25){\begin{tiny}
			$u$
			\end{tiny}}
		\put(8.4,-1){\begin{scriptsize}
			$ (b)$
			\end{scriptsize}}
		\put(5.5,3){\begin{scriptsize}
			$ T_xM$
			\end{scriptsize}}
			
		\draw (8.5,1.2) arc (-90:90:1);
		
		\end{tikzpicture}
			\bigskip
			\bigskip
			\bigskip
			\bigskip
			\bigskip
		\caption{The Finslerian indicatrix seen from the $h$-Riemanian perspective (a) and Finslerian perspective (b), respectively.}\label{fig: Kropina indicatrix, 2 cases}
	\end{center}
\end{figure}



From geometrical point of view, we observe that the indicatrix of a Kropina metric is obtained by a rigid translation of the $h$-Riemannian unit circle in the direction pointed by $W$. The resulting indicatrix
$S_{A_x}(1):=\{y\in A_x:F(x,y)=1\}$
 is an 
 $h$-Riemannian unit circle passing through the origin of $T_xM$, fact implying the conicity of the Kropina metric, see Figure \ref{fig: Kropina indicatrix, 2 cases} (a), where the norm in $T_xM$ is given by $h$.

If we consider the vector space $T_xM$ with the Finslerian norm $F$, then $S_{A_x}(1)$ looks like a fan in a half open subspace of $T_xM$, see Figure \ref{fig: Kropina indicatrix, 2 cases} (b). Observe that formula (\ref{eq for F}) implies 
$F(u)=\frac{1}{2\cos\theta}$, where 
$u\in A_x$, $||y||_h=1$, and $\theta\in (-\frac{\pi}{2},\frac{\pi}{2})$ is the angle between $u$ and $W_x$. In this case we obtain the polar equation $r=\frac{1}{2\cos\theta}$ that is actually the vertical line $x=\frac{1}{2}$.

\begin{remark}
The navigation data $(h,W)$, that is a Riemannian metric $h$ on $M$ and an $h$-unit vector field
$W$, induces a unique Kropina metric \eqref{Kropina metric def}, and conversely, starting with
a given Kropina metric \eqref{Kropina metric}, the Riemannian metric $a$ and the 1-form $\beta$
uniquely induce the navigation data \eqref{2.5}.

However, unlike the Randers case, this is not a one-to-one correspondence between the
navigation data $(h,W)$ and the $(\alpha,\beta)$-data. Indeed, given
the $(\alpha,\beta)$-data we get a unique navigation data $(h,W)$ by \eqref{2.5}.

Conversely, if starting with the navigation data $(h,W)$, then we obtain a family of corresponding
 $(\alpha,\beta)$-data \eqref{a_ij, b_i def}, depending on a  smooth function $\kappa$ on $M$.
\end{remark}


\subsection{The list of manifolds admitting globally defined Kropina metrics}

 The description given here of Kropina metrics as solutions of the Zermelo's navigation problem is correct only in the case $|W|_h=1$.
  Hence we obtain the following characterization.
 
 \begin{lemma}
 A differentiable manifold $M$ admits a globally defined Kropina metric if and only if it admits a nowhere vanishing vector field $X$.
 \end{lemma}
 
 Indeed, by normalization using a Riemannian metric $h$, we obtain the $h$-unit wind $W:=\frac{1}{\sqrt{h(X,X)}}X$ on $M$, and hence the navigation data $(h,W)$ that uniquely induces a Kropina metric.
 
  This requirement imposes immediately topological restriction on the manifolds admitting well defined Kropina metrics and the complete characterization of such manifold is studied in differential topology.

General notions in differential topology
allow us to conclude the following.
It is true that these are classical facts in topology, but since it is hard to find a self contained monograph containing all these details, we decided to write a systematic description here. 

We obtain the following list of manifolds admitting globally defined Kropina metrics.
\begin{theorem}\label{Globally Kropina}

\begin{enumerate}
\item Any compact connected manifold with boundary admits a globally defined Kropina metric.
\item Any compact connected boundaryless manifold $M$ admits a globally defined Kropina metric if and only if $\chi(M)=0$.
\item Any compact connected odd dimensional manifold (regardless it has boundary or not) admits a globally defined Kropina metric.
\item Any connected non-compact manifold admits a globally defined Kropina metric.
\item If $M$ admits a globally defined Kropina metric,
then the product manifold $M\times N$ also admits a globally defined Kropina metric.
\item Every Lie group $G$ admits $n=\dim G$ distinct globally defined Kropina metrics, one of each corresponding to one vector field in a parallelization of $G$.
\item If $M$ and $N$ are parallelizable, then $M\times N$ admits $m\times n$ distinct globally defined Kropina metrics, where $m$ and $n$ are the dimensions of $M$ and $N$, respectively.

\end{enumerate}

\end{theorem}

\begin{corollary}

\begin{enumerate}
\item For any manifold $M$, the topological cylinder $\R\times M$ admits a globally defined Kropina metric.
\item Any odd dimensional sphere admits a globally defined Kropina metric.
\item Any compact orientable 3-manifold admits a globally defined Kropina metric.
\end{enumerate}

\end{corollary}

\begin{remark}
Observe that the existence of a globally defined Kropina metric
is equivalent to the existence of a Lorentz type metric on $M$ since both conditions are actually equivalent to the existence of a nowhere vanishing vector field (see for instance \cite{ON}, Proposition 37 in chapter 5). 

We also recall that this condition reads that $M$ is 
\begin{enumerate}
	\item a compact manifold with vanishing Euler characteristic, or
	\item a non-compact manifold,
\end{enumerate}
(see \cite{V} and \cite{Kk}, respectively). 
\end{remark}


\section{Strong Kropina metrics}

\quad We will assume in the following that $(M,h)$ is a complete Riemannian manifold. Even though we don't need completeness of $h$ for all our results, we can assume completeness without loosing the generality. 

We recall that a smooth vector field $W$ on $(M,h)$ is called {\it Killing} if $\mathcal L_Wh=0$, where $\mathcal L$ is the Lie derivative, which is equivalent to
\[
W(h(Y,Z))=h([W,Y],Z)+h(Y,[W,Z])
\]
for all smooth vector fields $Y$, $Z$ on $M$.
 Observe that this relation is equivalent to the coordinates {\it Killing equations} (for the covariant components $W_i=h_{ij}W^j$)
\[
W_{i|j}+W_{j|i}=0,
\]
where $W=W^i\frac{\partial }{\partial x^i}$, and $\ _{|i}$ is the covariant derivative of $h$.

Equivalently, a smooth vector field $W$ on a complete Riemannian manifold $(M,h)$ is Killing if and only if the flow it generates (i.e. the local one-parameter group of local diffeomorphisms) $\Phi$ is complete and the transformations $\Phi_t$, $t\in \R$, are Riemannian isometries of $h$.


A special class of Kropina metrics are the strong Kropina metrics defined as follows.

\begin{definition}
The Kropina metric obtained from the navigation data $(h,W)$, where $W$ is an $h$-unit Killing vector field of $(M,h)$, is called a {\it strong Kropina metric}.
\end{definition}

Obviously, we have

\begin{lemma}
A manifold $M$ admits a strong Kropina metric if and only if on $M$ there exists a Riemannian metric $h$ with a unit Killing vector field W.
\end{lemma}

We recall a fundamental result about unit length Killing vector fields on a Riemannian manifold.

\begin{lemma}{\rm (\cite{BN1})}\label{integral lines geodesics}
Let $W$ be a Killing vector field on a Riemannian manifold $(M,h)$. Then the followings are equivalent:
\begin{enumerate}
 \item[{(i)}] $W$ has constant length on $(M,h)$.
 \item[{(ii)}] $W$ is auto-parallel on $(M,h)$, that is $\nabla^{(h)}_WW=0$, where $\nabla^{(h)}$ is the Levi-Civita connection of $h$.
  \item[{(iii)}] Every integral curve of the field $W$ is a geodesic of $(M,h)$.
  \end{enumerate}
\end{lemma}

Indeed, the main idea behind this result is the following elementary computation
\begin{equation*} 
(\mathcal L_Wh)(W,Y)=h(\nabla^{(h)}_WW,Y)+\frac{1}{2}Y(h(W,W))
\end{equation*}
for any vector field $Y$,  proving that (i) is equivalent to (ii).

Clearly, there are geometrical and topological obstructions on the manifold $M$ for the existence of strong Kropina metrics. A necessary condition is that the manifold $M$ is on the list of differentiable manifolds admitting globally defined Kropina metrics in Theorem \ref{Globally Kropina}.
Sufficient conditions must involve conditions on the Riemannian metric $h$.

Here are some manifolds that do not admit strong Kropina metrics.

\begin{theorem}\label{not possible strong Kropina}
	If $(M,h)$ is one of the followings
\begin{enumerate}
\item   an even-dimensional compact Riemannian of positive sectional curvature, or
\item   a Riemannian manifold with negative Ricci curvature,
\end{enumerate}
then $M$ does not admit strong Kropina metrics.
\end{theorem}

Indeed, it is known that  such Riemannian manifolds do not admit nowhere vanishing Killing vector fields (see \cite{BN1}, Theorem 5 and Corollary 2).

On the other hand, there are a lot of Riemannian manifolds admitting strong Kropina metrics.
One clear candidate to be a unit Killing vector field is a parallel vector field on the Riemannian manifold $(M,h)$. Indeed, we have
\begin{theorem}
Let $W$ be a non-trivial parallel vector field on the Riemannian manifold $(M,h)$. Then 
\begin{enumerate}
 \item[{(i)}] The manifold $M$ admits a strong Kropina metric $F$ induced by the navigation data $(h, W)$.
  \item[{(ii)}]  $M$ is a local direct metric product of some 1-dimensional manifold, tangent to the field $W$, and its orthogonal complement. Thus, the universal covering $\widetilde M$ of the Riemannian manifold $(M,h)$ is a direct metric product $\widetilde M=N\times \R$, where $N$ is an $(n-1)$-dimensional submanifold and the field $\widetilde W$ (that projects to $W$ under the natural projection $\widetilde M\to M$) is tangent to the $\R$ direction.
  \end{enumerate}
\end{theorem}

Indeed, it is easy to see that if $W$ is parallel then  we can make it a unit  Killing vector field on $(M,h)$.
The second part follows from the well-known de Rham decomposition Theorem (\cite{BN1}).

\begin{remark}
We point out that to be a unit Killing vector field is not equivalent to being parallel. The simplest example are the odd dimensional spheres that have unit Killing vector fields, but no parallel ones.
\end{remark}

    \begin{theorem}\label{thm: 1st properties}
    Let $M$ be a differential manifold admitting a strong Kropina metric $F$ induced by the navigation data $(h,W)$. Then
    \begin{enumerate}
    \item[{(i)}] $\Ric(W,W)\geq 0$, where $\Ric$ is the Ricci curvature of $(M,h)$.
   \item[{(ii)}] The equality $\Ric(W,W)=0$ is equivalent to $W$ being parallel on $(M,F)$.
    \end{enumerate}
    \end{theorem}
    
    This follows from \cite{BN1}, Theorems 3  and 4.

Moreover, we have

\begin{theorem}
If $(M,h)$ is a Riemannian manifold of non-positive sectional curvature and $W$ is a Killing vector field on $(M,h)$, then the followings are equivalent
\begin{enumerate}
 \item[{(i)}] $M$ has a strong Kropina metric induced by the navigation data $(h,W)$.
   \item[{(ii)}] $W$ is parallel on $(M,h)$.
\end{enumerate}
\end{theorem}

It follows from Proposition 7 in \cite{BN1}.

 In the case $M$ is compact, stronger results can be obtained.
 
 \begin{theorem}
 Let $(M,h)$ be a compact Riemannian manifold and $W$ a Killing vector field on $(M,h)$ such that $\Ric(W,W)\leq 0$. Then
  \item[{(i)}] $W$ is parallel on $(M,h)$, and $\Ric(W,W)=0$.
     \item[{(ii)}] $M$ has a strong Kropina metric induced by the navigation data $(h,W)$.
 \end{theorem}
 
 Indeed, it is known that on such a Riemannian manifold, $W$ is automatically parallel, and hence it has constant length (see \cite{YB}).

 \begin{corollary}
  If $(M,h)$ is a compact Riemannian manifold of non-positive Ricci curvature, then $M$ has a strong Kropina metric induced by the navigation data $(h,W)$, where $W$ is any Killing vector field on $(M,h)$.
 \end{corollary}
 
 Indeed, on such Riemannian manifolds any Killing vector field is parallel and hence of constant length (see \cite{Ko}).
 
\begin{remark}\label{Neg Sectional curv not allowed}
Observe that $h$-negative sectional curvature case is not allowed because such manifolds do not admit strong Kropina metrics. Indeed, for the navigation data $(h,W)$ on $M$, we recall that
 		\[
 		\Ric_x(W,W)=\frac{1}{n-1}\sum_{i=1}^{n-1}h(R(W,Z_i)W,Z_i)=\frac{1}{n-1}\sum_{i=1}^{n-1}K(W,Z_i),
 		\]
 		where $Z_1,\dots,Z_n$ is an $h$-orthonormal basis of $T_xM$ such that $Z_n=W$, $R$ and $K$ are the curvature tensor and the sectional curvature of $h$, respectively.
 		
 		Then the $h$-negative sectional curvature implies $\Ric(W,W)<0$, but this is not possible (see Theorem \ref{thm: 1st properties}).
\end{remark}

  \begin{proposition}\label{integrable_distrib}
  If $(M,F)$ is a strong Kropina metric with navigation data $(h,W)$. Then the followings are equivalent
  \begin{enumerate}
  \item[{(i)}] The 1-form $W^{\flat}$ is closed on $M$,  where $W^{\flat}$ is the Legendre dual of $W$ with respect to $h$.
  \item[{(ii)}]  $W$ is parallel.
    \item[{(iii)}] The $h$-orthogonal vector distribution $W^{\perp}:=\cup_{x\in M}W^{\perp}_{x}$ is involutive, where  $W^{\perp}_{x}:=\{X\in T_{x}M:h(X,W)=0\}$.
    \end{enumerate}
  \end{proposition}
  \begin{proof}

 (i)$\Rightarrow$ (ii)
 
 If $(M,F)$ is a strong Kropina metric with navigation data $(h,W)$, then the conditions that $W^{\flat}$ is closed 
  and $W$ is Killing field give
  \begin{equation}
  \begin{cases}
  \frac{\partial W_{i}}{\partial x^{j}}-\frac{\partial W_{j}}{\partial x^{i}}=0,\\
  \frac{\partial W_i}{\partial x^j}+\frac{\partial W_j}{\partial x^i}-2\gamma_{ij}^kW_k=0,
  \end{cases}
  \end{equation}
   { for all } $i,j,k\in\{1,2,\dots,n\}$, $W=W^{i}\frac{\partial}{\partial x^{i}}$, and $W_{i}=h_{ij}W^{j}$,  where $\gamma_{ij}^k$ are the Christoffel symbols of the Riemannian metric $h$.
  From here we get $\frac{\partial W_i}{\partial x^j}-\gamma_{ij}^kW_k=0$, that is $W$ is parallel.
  
  (ii)$\Rightarrow$ (i)

  If $W$ is parallel on $(M,h)$, that is
  \[
  \frac{\partial W_{i}}{\partial x^{j}}-\gamma_{ij}^{k}W_{k}=0, \textrm{ for all } i,j\in\{1,2,\dots,n\},
  \]
   then by interchanging $i$ and $j$ and subtraction we obtain
  \[
  \frac{\partial W_{i}}{\partial x^{j}}-\frac{\partial W_{j}}{\partial x^{i}}=0, \textrm{ for all } i,j\in\{1,2,\dots,n\},
  \]
and hence $W^{\flat}$ is closed. 

  (i)$\Leftrightarrow$ (iii)
  The use of Frobenius theorem proves this equivalence.
  $\qedd$
  
  \end{proof}
  
  In the 2-dimensional case, a more detailed classification of manifolds admitting strong Kropina metrics is possible.

 Let us  recall the following definition
    \begin{definition}\rm{(\cite{BN1})}\label{def: quasi-regular}
    	
    	Let $W$ be   a complete Killing vector field of constant length
on the Riemannian manifold $(M,h)$. Then $W$ is called
    	\begin{enumerate}
    		\item[{(i)}] {\it quasi-regular} if all integral curves of $W$ are closed (there might exist closed curves of different $h$-lengths);
    		\item[{(i)}] {\it regular} if it is quasi-regular and all integral curves have the same $h$-lengths.
    	\end{enumerate}
    

    	
    \end{definition}
    
  Then we obtain the following classification theorem.

  \begin{theorem}\label{thm: topological classification in two dimensions}
  Let $M$ be a 2-dimensional manifold admitting a strong Kropina metric induced by the navigation data $(h,W)$. Then
  \begin{enumerate}
 \item[{(i)}] The Riemannian surface $(M,h)$ is flat, i.e. $M$ is isometric to one of the manifolds: Euclidean plane, straight cylinder (in the non-compact case), or flat torus, M\"obius band, Klein bottle (in the compact case).
 \item[{(ii)}] In the cases $M$ isometric to M\"obius band, Klein bottle, $W$ is quasi-regular.
 \end{enumerate}
  \end{theorem}
  
  This follows from Corollary 6 in \cite{BN1}.

  A stronger condition is the Killing property.
  
  \begin{definition}
  A Riemannian manifold $(M,h)$ is said to have the {\it Killing property} if, in some neighborhood $U$ of each point $x\in M$, there exists an $h$-orthonormal frame $\{X_1,\dots,X_n \}$ such that each $X_i$, $i\in\{1,2,\dots,n\}$, 
  is a Killing field. Such a frame is called a {\it Killing frame}.
  
  If the Killing frame is defined globally on $M$, then $(M,h)$ is called {\it manifold with the global Killing property}.
  \end{definition}
  
  \begin{remark}
  \begin{enumerate}
  \item A global Killing frame defines a parallelization of $(M,h)$.
  \item Every Lie group with bi-invariant Riemannian metric has the global Killing property.
  \end{enumerate}
  \end{remark}
  

 From Proposition 9 in \cite{BN2} we obtain
 
 \begin{theorem}
 Let $(M,h)$ be an $n$-dimensional simply connected, complete Riemannian manifold with the local Killing property.
 Then
 \begin{enumerate}
  \item[{(i)}] $(M,h)$ has the global Killing property.
  \item[{(ii)}] $(M,h)$ is a symmetric Riemannian manifold.
  \item[{(iii)}] Each of the linearly independent unit Killing fields on $M$ induces a  strong Kropina structure on $M$.
  \end{enumerate}
\end{theorem}

  Observe that due to the Killing property there are infinitely many unit Killing fields on $M$ and each of them induces a  strong Kropina structure.

  Remarkably, in this case a complete classification is also possible (see Theorem 11 in \cite{BN2}).
  
  \begin{theorem}\label{thm: topological classification in the complete case}
  Let $(M,h)$ be a simply connected, complete Riemannian manifold. Then the followings are equivalent
  \begin{enumerate}
    \item[{(i)}] $M$ admits $n$ distinct
      strong Kropina structures induced from the navigation data $(h,W_i)$, where $\{W_1,\dots,W_n\}$
      is a global Killing frame.
    \item[{(ii)}] M is isometric to a direct product of Euclidean spaces,
    compact connected simple Lie groups with bi-invariant metrics, 7-dimensional
    round spheres (some factors can be absent).
  \end{enumerate}
  \end{theorem}
  
  \section{The geodesics and exponential map of strong Kropina spaces}
  \subsection{The geodesics equation}
  \hspace{0.2in}  In this section, we suppose that the Kropina space $(M,F:A\to \R^+)$  is a globally defined strong Kropina space, that is, the vector field $W$ is  a unit Killing vector field with respect to $h$.
  The existence of such a vector field imposes topological restriction on $M$ as seen in the previous section.

  For any point $p$ of $M$, there exist a neighborhood $U$ of $p$, a positive number $\epsilon$ and a local $1$-parameter group of local $h$-isometries $\varphi_t : U \longrightarrow M$, 
  $\varphi_t(x)=\varphi(t,x)$, 
  $t\in I_\epsilon=(-\epsilon, \epsilon)$ which induces the given $W$.

  Given the $h$-arclength parameterized Riemannian geodesic $\rho:(-\varepsilon,\varepsilon)\to M$,
  with initial conditions $\rho(0)=p$, $\dot{\rho}(0)=v$, we consider the curve
  \begin{equation}\label{def of P}
  \mathcal P(t) :=\varphi_t(\rho(t)), \textrm{ for any } t\in (-\varepsilon,\varepsilon).
  \end{equation}
  
  By derivation we get
  \begin{equation}\label{P dot}
    \dot{\mathcal P}(t) =W_{\mathcal P(t)}+{(\varphi_t)}_{*\rho(t)}(\dot{\rho}(t)), \textrm{ for any } t\in (-\varepsilon,\varepsilon),
    \end{equation}
    in particular, $t=0$ gives
    \begin{equation}
        \dot{\mathcal P}(0) =W_{p}+v,
        \end{equation}
        and hence we obtain $\dot{\mathcal P}(0)\neq 0$ if and only if $v\neq -W_{p}$.

  \begin{lemma}\label{lem: non-degenerate curve}
  If $\rho:(-\varepsilon,\varepsilon)\to M$ is an $h$-geodesic on $M$ with initial  conditions   
   \begin{equation}\label{rho initial cond}
   \rho(0)=p,\quad \dot{\rho}(0)=v\neq -W_p,
   \end{equation}
   then the curve $\mathcal P(t)$ given by
  \eqref{def of P} is non-degenerate, i.e. $\dot{\mathcal{P}}(t)\neq 0$ for all
  $t\in (-\varepsilon,\varepsilon)$.
  
  \end{lemma}
  \begin{proof}
  	Since from (\ref{P dot}) we have $\dot{\mathcal P }(t)={(\varphi_t)}_{*\rho(t)}(W_{\rho(t)}+\dot{\rho}(t))$, we have only to show that $W_{\rho(t)}+\dot{\rho}(t)\ne 0$ for all $t\in (0, \epsilon)$.

  	 Let us suppose that there exists a $t_0>0$ such that $\dot{\rho}(t_0)$ and $-W_{\rho(t_0)}$
  	 coincide as vectors of the vector space $T_{\rho(t_0)}M$.
  	  	 	Since $W$ is a unit  Killing vector field, the curve $\varphi_{-(t-t_0)}(\rho(t_0))$ is an $h$-unit speed geodesic passing through $\rho(t_0)$ with the tangent vector $-W_{\rho(t_0)}$ at $\rho(t_0)$.
  	 	Therefore, we get $\rho(t)=\varphi_{-(t-t_0)}(\rho(t_0))$, and from here it follows
  	 	$\dot{\rho}(t)=-W_{\varphi_{-(t-t_0)}(\rho(t_0))}$, for any $t$.
  	 	Putting $t=0$, we have  $\dot{\rho}(0)=-W_{\varphi_{t_0}(\rho(t_0))}=-W_p$  because of $\varphi_{t_0}(\rho(t_0))=p$.
  	 	This is a contradiction with \eqref{rho initial cond}.
  	  $\qedd$
  \end{proof}
  
  If $\rho$ is an $h$-geodesic, subject to the initial conditions
  \eqref{rho initial cond} and $\mathcal P(t)$ is the curve given by \eqref{def of P}, on the strong Kropina manifold $(M,F:A\to \R^+)$,
    then we will prove that
    \begin{enumerate}
    \item[{(i)}] $\mathcal P(t)$ is an $F$-admissible curve, i.e. $\dot{\mathcal P}(t)\in A_{\mathcal P(t)}$, for all $t$;
    \item[{(ii)}] $F(\mathcal P(t),\dot{\mathcal P}(t))=1$, i.e. $\mathcal P$ is $F$-arclength
    parametrized when we use the same parameter as on $h$;
    \item[{(iii)}] $\mathcal P(t)$ is a geodesic of the strong Kropina metric $F$ induced by the
    navigation data $(h,W)$.
    \end{enumerate}
  
  The first claim is easy to prove. Due to \eqref{conic domain by inequality} we need to show that $h(\dot{\mathcal P}(t),W_{\mathcal P(t)})>0$. Indeed, relation \eqref{P dot} implies
  \begin{equation*}
  \begin{split}
   h(\dot{\mathcal P}(t),W_{\mathcal P(t)}) &=|W_{\mathcal P(t)}|_h^2
+
    h({(\varphi_t)}_{*\rho(t)}(\dot{\rho}(t)),W_{\mathcal P(t)})=1+    h(\dot{\rho}(t)_{\rho(t)},W_{\rho(t)})   \\
   & =1+ \cos (\dot{\rho}(t)_{\rho(t)},W_{\rho(t)})     >0,
  \end{split}
  \end{equation*}
  where we have used that $\varphi_t$ is an isometry and Lemma \ref{lem: non-degenerate curve}. So (i) is proved.
  
 Using again \eqref{P dot} and Remark \ref{rem: equiv}, the claim (ii) is obvious.
 
 Finally, one can show that the curve \eqref{def of P} is indeed an $F$-geodesic by some local computation similar with the Randers case (see \cite{R}),  
 (see Theorem 1.2  in \cite{JV} for a more general case).

  It follows that, if
 $\rho:(-\varepsilon,\varepsilon)\to M$ is an $h$-arclength parametrized
 geodesic with initial conditions \eqref{rho initial cond}, then the curve $\mathcal P$
 given in \eqref{def of P} is an $F$-arclength parametrized Kropina geodesic.
 
 Conversely, suppose that a curve $\mathcal P:(-\varepsilon,\varepsilon)\to M$ is an
 $F$-arclength parametrized Kropina geodesic, that is $\mathcal P$ satisfied conditions (i),
 (ii), (iii) above. Then, similarly with our discussion above, one can see that the curve
 $\rho:(-\varepsilon,\varepsilon)\to M$, $\rho(t)=\varphi_{-t}( \mathcal{P}(t))$ is an
 $h$-arclength Riemannian geodesic on $M$ with initial conditions \eqref{rho initial cond}.

 Therefore, we obtain
 \begin{theorem}\label{Theorem 3.9}
 Let $(M, F=\alpha^2/\beta : A\to\R^+)$ be a globally defined strong Kropina space with  navigation data $(h, W)$.
 Let $p$ be any point of $M$, and let $\{\varphi_t\}_{t\in I_\epsilon}$ be a local $1$-parameter group of local  isometries on $(M, h)$ which generate the unit vector field $W$.
 
 Then the following two statements hold:
 \begin{enumerate}[(i)]
  \item  If the curve $\rho(t)$ is an $h$-geodesic on $(M, h)$  parameterized as $|\dot{\rho}(t)|_h=1$ and satisfies the conditions $\rho(0)=p$ and $\dot{\rho}(0)\ne -W(\rho(0))$, then the curve $\mathcal P(t):=\varphi_t( \rho(t))$, $t\in I_\epsilon$, is an $F$-geodesic of the Kropina space $(M, F=\alpha^2/\beta : A\to\R^+)$ which is parameterized as $F(\mathcal P(t), \dot{\mathcal P}(t))=1$.
 \item   If  the curve $\mathcal P(t)$, $t\in I_\epsilon$, is an $F$-geodesic on the Kropina space $(M, F)$  parameterized as $F(\mathcal P(t), \dot{\mathcal P}(t))=1$, the curve $\rho(t):=\varphi_{-t}( \mathcal P(t))$, $t\in I_\epsilon$, is an $h$-geodesic on $(M, h)$ which is parameterized as $|\dot{\rho}(t)|_h=1$.
 \end{enumerate}
 \end{theorem}

 
 Some further computations give us conditions for the Kropina metric $F=\frac{\alpha^2}{\beta}$ to be projectively equivalent to the Riemannian structure $a$.

 \begin{proposition}\label{Proposition 4.4}
 	Let $(M,F:A\to \R^+)$  be a Kropina manifold. Then the followings are equivalent
 	\begin{enumerate}
 		\item 	$F$ is projectively equivalent to the Riemannian structure $a=(a_{ij})$;
 		\item the covariant vector $b(x)=(b_1(x),\dots, b_n(x))$ is parallel with respect to the Riemannian metric $a$;
 		\item $\kappa=$ constant and the covariant vector $W(x)=(W_1(x),\dots,W_n(x))$ is parallel with respect to $h$, where the indices of $W$ are lowered by means of $h$.
 	\end{enumerate}
If any of the above statements hold, then the Kropina space $(M,F:A\to \R^+)$ is strong.
 \end{proposition}
 
The proof is merely a local computation that can be found in Appendix. The final statement follows from Lemma \ref{integral lines geodesics}.

%
 
 \subsection{The exponential map}

 \hspace{0.2in} Let $(M, F=\alpha^2/\beta : A\to\R^+)$ be a Kropina (not necessarily strong) space and let $x\in M$.

 Let $c_y:(-\epsilon, \epsilon)\to M$ be an $F$-geodesic with the  conditions $c(0)=x$ and $\dot{c}(0)=y\in A_x$.
   By  homogeneity, one has $c_y(\lambda t)=c_{\lambda y}(t)$ for any $\lambda>0$, and therefore
 if we choose a suitable $\lambda$, for example  $\lambda=\epsilon/2$, the curve $c_{\lambda y}(t)$ is defined at $t=1$.
 Hence, there exists a neighborhood $U$ of $0$ in $T_xM$  and a mapping
  $  A_x \cap U  \longrightarrow M, \hspace{0.2in}
         y   \longmapsto  c_y(1)$.

 Now, we will give the following definition:
 \begin{definition}
 Let $(M, F=\alpha^2/\beta : A\to\R^+)$ be a Kropina  space. A Kropina metric $F$  is said to be {\it forward (resp. backward) complete} if every Kropina geodesic $c(t)$ on $I=(-\epsilon, \epsilon)$  can be extended to a Kropina geodesic defined on $(-\epsilon, \infty)$ (resp. $ (-\infty, \epsilon)$).
 $F$ is said to be {\it bi-complete} if it is both  forward complete and  backward complete.
 \end{definition}

 We define the exponential  map $\exp_x$  as follows.

 \begin{definition}
 The mapping $\exp_x : A_x \longrightarrow M$ defined by
 \begin{eqnarray*}
  \exp_x(y):=c_y(1)
 \end{eqnarray*}
 is called   the {\it Kropina exponential map} at $x$.
 In the case $F$ forward complete, $\exp_x{(t y)}$ is defined for $t\in (0, \infty)$.
 \end{definition}
 
  Clearly, we have
 \[c_y(t)=c_{ty}(1)=\exp_x(t y), \hspace{0.2in} t>0.\]
 
 Furthermore, we extend the  Kropina exponential map $\exp_x$ at $x$.
 \begin{definition}
  The map $\overline{\exp}_x : A_x \cup \{0_x\}  \longrightarrow M$ is defined by
 \begin{eqnarray*}
 \overline{\exp}_x(y)={\exp_x(y),\hspace{0.1in}\ (y\in A_x) ,
                       \atopwithdelims\{.x,      \hspace{0.4in} (y=0_x).}
 \end{eqnarray*}
 We call the Kropina  geodesic $\overline{\exp}_x{(t y)}$,    {\it a radial geodesic}.
 \end{definition}
 
 Due to $c_y(0)=x$, we have $c_y(t)=\overline{\exp}_x(t y)$, $t\ge 0$.

 Let $y\in A_x$ and $\lambda >0$.
 Since $(\exp_x)^i{(\lambda y)}=c^i_y(\lambda)$, we have
 \begin{eqnarray*}
 \frac{\partial (\exp_x)^i}{\partial y^k}(\lambda y) \lambda {\delta^k}_j=\frac{\partial c^i_y}{\partial y^j}(\lambda),
 \end{eqnarray*}
 that is,
 \begin{eqnarray*}
 \frac{\partial (\exp_x)^i}{\partial y^j}(\lambda y)  =\frac{1}{\lambda}\frac{\partial c^i_y}{\partial y^j}(\lambda).
 \end{eqnarray*}
 Note that $c^i_y(0)=x$ for all $y\in A_x$, and hence
 \begin{eqnarray*}
         \frac{\partial c^i_y}{\partial y^j}(0)=0.
 \end{eqnarray*}
 Therefore, we have
 \begin{equation*}
 \frac{\partial (\exp_x)^i}{\partial y^j}(\lambda y)  =\frac{1}{\lambda}\frac{\partial c^i_y}{\partial y^j}(\lambda)
                                                            =\frac{1}{\lambda}\bigg(\frac{\partial c^i_y}{\partial y^j}(\lambda)-\frac{\partial c^i_y}{\partial y^j}(0)\bigg).
 \end{equation*}
 Letting $\lambda \longrightarrow +0$, we have
 \begin{eqnarray*}
    \lim_{\lambda \longrightarrow +0}\frac{\partial (\exp_x)^i}{\partial y^j}(\lambda y)
 =  \lim_{\lambda \longrightarrow +0}\frac{1}{\lambda}\bigg(\frac{\partial c^i_y}{\partial y^j}(\lambda)-\frac{\partial c^i_y}{\partial y^j}(0)\bigg)
 =  \frac{\partial^2 c^i_y}{\partial y^j\partial t}(0)
 =\frac{\partial}{\partial y^j}(y^i)\
 ={\delta^i}_j,
 \end{eqnarray*}
 and hence
 \begin{eqnarray*}
        \lim_{\lambda \longrightarrow +0}\frac{\partial (\overline{\exp}_x)^i}{\partial y^j}(\lambda y)
                    = \lim_{\lambda \longrightarrow +0}\frac{\partial (\exp_x)^i}{\partial y^j}(\lambda y)       ={\delta^i}_j.
 \end{eqnarray*}

 Next, we will consider the exponential map of a strong Kropina space $(M, F=\alpha^2/\beta : A\to\R^+)$ with   navigation data $(h, W)$.
 Using Theorem \ref{Theorem 3.9} we will relate the Kropina  exponential map to the Riemannian one. 
 Let us denote by $e_x : T_xM \longrightarrow M$ the usual exponential map of $(M, h)$.
 
 Recall that, for any  $y\in A_x\subset T_xM$, $F(x,y)=1$, we have $|y-W(x)|_h=1$, and therefore
 \begin{equation}\label{F and h-exponential maps}
  \exp_x{(ty)}=\varphi_{ t  }  \circ e_x\bigg(t(y-W(x))\bigg).
 \end{equation}

 \begin{eqnarray*}
                                                           &exp_x&\\
                                       A_x   \hspace{0.05in}&\longrightarrow &    M\\
       {\it(-W)}-{\it translation} \hspace{0.1in}    \bigg{\downarrow}   \hspace{0.1in}  &  \hspace{1in}&\bigg{\uparrow}     \varphi_{t}\\
                                      T_xM  &\longrightarrow & M\\
                                               &e_x&
 \end{eqnarray*}
 
 (see Figures \ref{figure 1}, \ref{figure 2} below).

 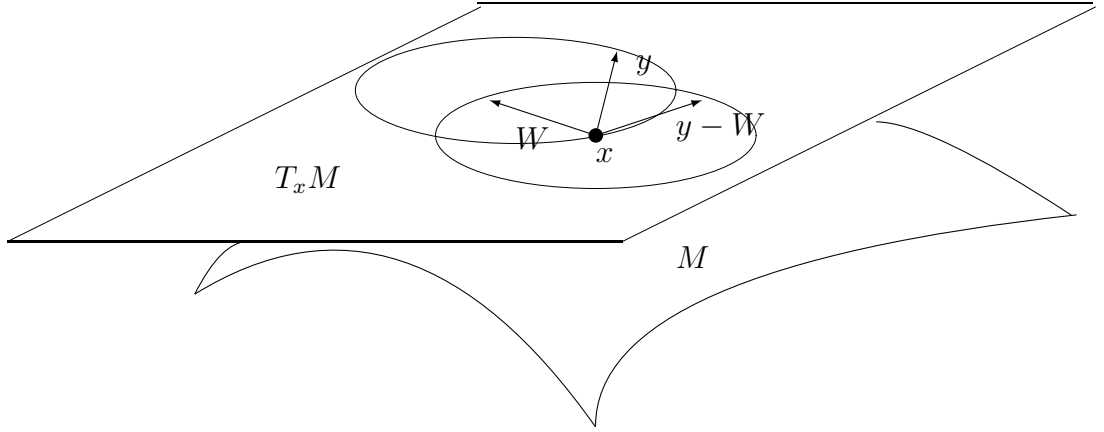
\begin{figure}[H]
 \begin{picture}(250,220)
              \put(120,80){\qbezier (0,.100) (80,50) (150,-50)}
             \put(170,30){\qbezier (100,.20) (100,60) (280,80)}
             \put(275,145){\qbezier (100,.20) (120, 0) (173,-35)}
           \put(120,80){\qbezier (0,.100) (10, 20) (20,20)}
              \put(50,100){\line(1,0){230}}
              \put(50,100){\line(2,1){177}}
              \put(226,190){\line(1,0){230}}
              \put(280,100){\line(2,1){177}}
              \put(270,140){\circle*{5}}
              \put(270, 140){\vector(3,1){40}}
               \put(270,140){\vector(-3,1){40}}
                \put(270, 140){\ellipse{120}{40}}
               \put(270, 140){\vector(1,4){8}}
                \put(240,157){\ellipse{120}{40}}
               \put(300, 90){$M$}
              \put(150, 120){$T_xM$}
              \put(270, 130){$x$}
               \put(300,140){$y-W$}
                \put(240, 135){$W$}
               \put(285, 165){$y$}
 \end{picture}
 \caption{The rigid displacement of the $h$-unit tangent sphere gives the Kropina indicatrix}
 \label{figure 1}
 \end{figure}%
 
 \vspace{2cm}
 
 \begin{figure}[H]
 
 \begin{picture}(200,80)
        \put(250,50){\ellipse{100}{60}}
      \put(285,20){$S_x^{h}M:=\{y\in T_xM : ||y||_h)=1\}$}
     \put(226,76){\ellipse{100}{60}}
     \put(30,80){$S_xM:=\{y\in T_xM : F(y)=1\}$}
    \put(325, 130){$\mathcal P(t):=\varphi_t(\rho(t))$}
     \linethickness{0.2pt}
       \put(250,50){\vector(2,1){39}}
    \linethickness{2pt}
     \put(250,50){\vector(-2,3){19}}
     \linethickness{4pt}
     \put(250,50){{\vector(1, 3){15}}}
     \put(285,75){$y-W$}
     \put(220,48){$W$}
     \put(265,70){${y}$}
     
    \linethickness{0.1pt}
     \put(150,50){\vector(1,0){200}}
    
    \linethickness{0.1pt}
    \put(250,0){\vector(0,1){120}}
  
  \put(250,50){\qbezier (0,.90) (50,20) (100,10)}
     \put(299,63){\circle*2}
     \put(280, 50){$e_x (y-W)$}
     \put(350, 60){$\rho(t):=\varphi_{-t}(\mathcal P (t))$}
   \put(250,50){\qbezier (0,.200) (30,80) (100,100)}
    \put(275,97){\circle*2}
   \put(280,97){$exp_xy$}
   \end{picture}\\
 \caption{The exponential map of a strong Kropina space vs the exponential map of the associated $h$-Riemannian space.}\label{figure 2}
 \end{figure}
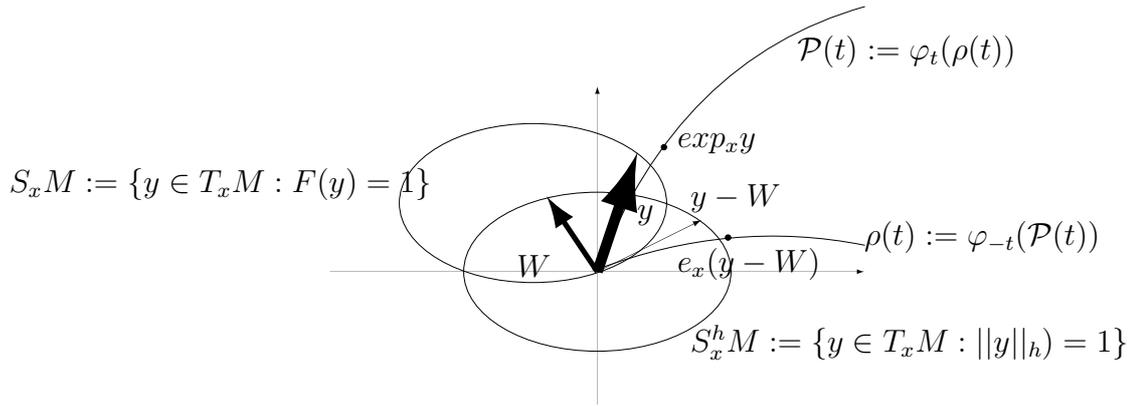%

 \subsection{Kropina metric balls}
 
 We recall some definitions from \cite{JS}.
 
 If $c:[a,b]\to M$ is an $F$-admissible curve on a Kropina manifold $(M,F=\alpha^2/\beta :A\to \R^+)$, that is $\dot{c}(t)\in A_{c(t)}$, $t\in [a,b]$, then the $F$-$length$ of $c$ is defined by
 \begin{equation*}
 \mathcal L_F(c|_{[a,b]})=\int_a^b F(c(t),\dot{c}(t))dt.
 \end{equation*}
 
 \begin{definition}
 	Let $(M,F)$ be a Kropina manifold, and $p,q\in M$. It is said that $p\ precedes\ q$, denoted $p\prec q$, if there exists an $F$-admissible piecewise $C^\infty$-curve from $p$ to $q$.
 \end{definition}
 
 One can define now the {\it forward} and {\it backward $F$-connected domain} of $p$ by
 \[
 \mathcal D_p^+:=\{q\in M\ :\ p\prec q\}\ \textrm{ and }  \mathcal D_p^-:=\{q\in M\ :\ q\prec p\},
 \]
 respectively, that is the set of points $q\in M$ that can be joined to $p$ with $F$-admissible piecewise $C^\infty$-curves from $p$ to $q$, or from
 $q$ to $p$, respectively. We have preferred this naming to {\it future} and {\it past} of $p$ in   \cite{JS} because it clarifies the geometrical meaning.

 
 \begin{definition}\label{F-separation}
 Let $p,q\in M$ be two arbitrary points on $M$ and let $\Gamma_{pq}^F$ be the set of $F$-admissible piecewise $C^\infty$-curves from $p$ to $q$. We define the $F$-
 {\it separation} $d_F:M\times M\to [0,\infty]$ from $p$ to $q$
 by
  \begin{equation*}
 d_F(p,q)=
 \begin{cases}
  \inf_{c\in \Gamma_{pq}^F}\mathcal L_F(c),\quad \textrm{ if } \Gamma_{pq}^F\neq \emptyset\\
  0,\ \qquad\qquad \qquad\textrm{ if } p=q\\
  \infty, \qquad \qquad \qquad \textrm{otherwise }.
   \end{cases}
  \end{equation*}
   \end{definition}
   
Even though it can take infinite values, the $F$-separation can have similar properties with the distance function of a classical Finsler space (called a quasi-metric in \cite{SSS}).

\begin{remark}
Despite of the similarities of a separation function with a Finslerian quasi-metric, there are essential differences as follows.
Notice that, in general, $d_F$ will take the infinite value for some pair of points and will also have some discontinuities (see \cite{CJS}, Theorem 4.5). 


\end{remark}

 We can define the $d_F$-metric {\it forward} and {\it backward balls} at a point $p\in M$ by
 \begin{equation*}
 \begin{split}
 & \mathcal{B}^+_F(p,r)=\{q\in M\ :\ \Gamma_{pq}^F\neq \emptyset,\ d_F(p,q)<r\} \textrm{  and}\\
 & \mathcal{B}^-_F(p,r)=\{q\in M\ :\ \Gamma_{pq}^F\neq \emptyset,\ d_F(q,p)<r\},
 \end{split}
 \end{equation*}
 respectively.

 \begin{remark}(\cite{JS})
 	\begin{enumerate}
 		\item On a strong Kropina space, the binary relation $\prec$ is transitive.
 		\item For any $p\in M$, the sets  $\mathcal D^+_p$ and $\mathcal D^-_p$ are open subsets of $M$.
 	\end{enumerate}
 
 \end{remark}

 \begin{remark}
 		\item Observe that for any  $p\in M$, locally, always there exists a point $q\in M$ such that $p$ can be joined to $q$ by a short $F$-geodesic. Therefore, $\mathcal{D}_p^+\neq\emptyset$, for any $p\in M$. Similarly $\mathcal{D}_p^-\neq\emptyset$, for any $p\in M$. We point out that this property is not related to the completeness of $(M,F)$ since our considerations here are essentially local ones.
 \end{remark}
 
 If $(M,F)$ is a strong Kropina manifold with navigation data $(h,W)$, then taking into account Remark \ref{rem: equiv} and Theorem \ref{Theorem 3.9}, we can see that, for  any two points $p,q\in M$, $q\in \mathcal{D}_p^+$, we have
 \[
 d_F(p,q)=d_h(p,q^-)=d_h(q,q^-),
 \]
 where $q^-=  \varphi_{-l}(q) $, and $l=d_F(p,q)$.
 
 \begin{remark}
 It can be seen that there exists a relation between the $h$-length and the $F$-length of a vector $y\in A\subset TM$. Indeed, due to the Zermelo's navigation construction presented in Section \ref{sec: Zermelo Navigation}, the tangent Kropina ball $B_p^F(1):=\{y\in T_pM : F(y)<1 \}$ is naturally included in the $h$-tangent ball $B_p^h(2)=B_p^{\frac{1}{4}h}(1)$, i.e. $B_p^F(1)\subseteq B_p^{\frac{1}{4}h}(1)$, that is 
 \begin{equation}\label{lower bounded F}
 F(y)\geq \frac{1}{2}||y||_{h}=||y||_{\frac{1}{4}h},
 \end{equation}
  for any $y\in A\subset TM$ (this is called in \cite{JS} a Riemannian lower bound for $F$, see same reference for a theory of more general lower bounded conic Finsler spaces).

 
 %

 Also observe that for any points $p$, $q\in M$, we have $\frac{1}{2}d_h(p,q)\leq d_F(p,q)$. Indeed, if $d_F(p,q)=\infty$, then the inequality is trivial. In the case $d_F(p,q)$ is finite, for any curve $\gamma\in \Gamma_{p,q}$ from $p$ to $q$,  relation \eqref{lower bounded F} implies $\frac{1}{2}\mathcal L_h(\gamma)\leq  \mathcal L_F(\gamma)$, where $\mathcal L_h(\gamma)$ and  $\mathcal L_F(\gamma)$ are the integral lengths of $\gamma:[a,b]\to M$ with respect to $h$ and $F$, respectively. Then, by using the definition of the infimum, for any $\varepsilon>0$, there exists a curve $\gamma_\ve\in \Gamma_{p,q}$ such that 
 \[
 d_F(p,q)> \mathcal L_F(\gamma_\ve)-\ve \geq \frac{1}{2}\mathcal L_h(\gamma_\ve)
 -\ve \geq \frac{1}{2}d_h(p,q)-\ve,
 \]
 and the desired formula follows for $\ve \to 0$.
 
 Moreover, $\mathcal{B}_F^+(p,\ve)\subseteq \mathcal{B}_h(p,2\ve)$ (see Figure \ref{fig: F geodesic ball vs h half geodesic ball}). Indeed, for any $x\in \mathcal{B}_F^+(p,\ve)$, 
that is $d_F(p,x)<\ve$, we have
 \[
  \frac{1}{2}d_h(p,x)\leq d_F(p,x) <\ve
 \]
 and hence $x\in \mathcal{B}_h(p,2\ve)$.



 \begin{figure}[h]
 	\begin{center}
 		\setlength{\unitlength}{1cm}
		\begin{picture}(9,6)
		\centering
 		\begin{tikzpicture}
 		\draw (0,0) ellipse (45mm and 30mm);
		\draw (-1,0) ellipse (20mm and 15mm);
 		\draw (1.75,0) ellipse (27.5mm and 15mm);
 		\end{tikzpicture}
 		\put(-5.55,3){{\tiny x}}
 		\put(-5.8,3){$p$}
		\put(-7,1){$\mathcal{B}_{h}(p, \varepsilon) $}
 		\put(-2,4){$\mathcal{B}_{F}^+(p, \varepsilon) $}
 		\put(-9.5,5){$\mathcal{B}_{h}(p, 2\varepsilon) $}
 		\end{picture}
 		\caption{The Kropina geodesic ball $\mathcal{B}_{F}^+(p, \varepsilon) $ vs the  $h$-geodesic ball $\mathcal{B}_{h}(p, 2\varepsilon) $.}\label{fig: F geodesic ball vs h half geodesic ball}
 	\end{center}
 \end{figure}
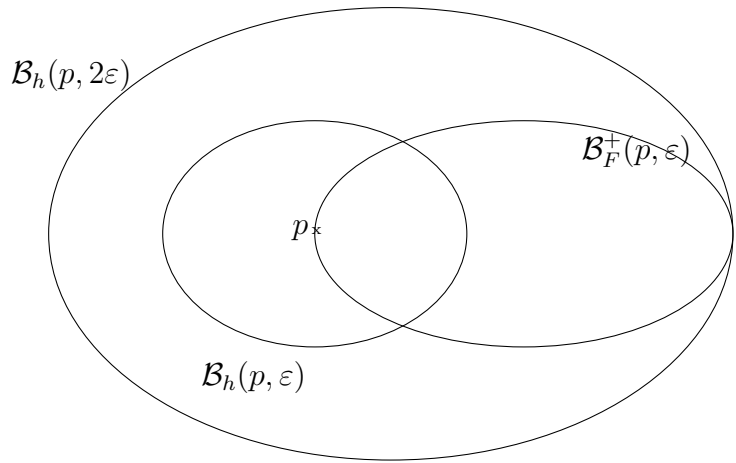


\bigskip 

 \end{remark}

 The following properties of the Kropina distance follow naturally (compare with the more general case of an arbitrary conic metric \cite{JS}).
 \begin{theorem}
  Let $(M,F)$ be a strong Kropina manifold with navigation data $(h,W)$. Then we have
 \begin{enumerate}[(i)]
  \item The forward and backward $d_F$-balls are open subsets of $M$.
 \item The collection of $d_F$-metric open forward  (respectively backward) balls  
forms a topological basis of $M$.
 \item  For sufficient small $r>0$, the tangent balls  ${ B_{A_p}(r) \cup }S_{A_p}(r)=\{y\in A_p : F(p,y)\leq r\}$ and the  {$d_F$-metric } balls $\mathcal{B}^+_F(p,r)$ are mapped each other by the exponential map $\overline{\exp}_p$.

 \end{enumerate}
 \end{theorem}

Indeed, it is easy to see that if we consider any fixed $p\in M$ and $r>0$,  then for any $x\in \mathcal{B}_h(p,2r)$, there exists two points $q, q' \in \mathcal{B}_h(p,2r)$ and some $\ve>0$ such that 
\begin{enumerate}
	\item $x\in \mathcal{B}^+_F(q,\ve)\cap \mathcal{B}^-_F(q',\ve)$, and
	\item  $\mathcal{B}^+_F(q,\ve)\cup \mathcal{B}^-_F(q',\ve)\subseteq \mathcal{B}_h(p,2r) $,
\end{enumerate} 
(see \cite{JS}). The statements in the theorem follows immediately. Observe that the exponential map is surjective in this case.

 \section{The minimality of Kropina geodesics}
 \subsection{Short geodesics}
 \hspace{0.2in} To begin with, we will formulate and prove the  Gauss Lemma for Kropina spaces.
 
 \begin{lemma}{\bf (The Gauss Lemma, Kropina version)}\label{Lemma 3.22}
 Let $(M, F=\alpha^2/\beta:A\to \R^+)$ be a Kropina space. Let $x\in M$, fix $y\in A_x$ and set $r=F(x, y)$.
 Suppose that $\exp_x{y}$ is defined  in A.
 Let $T\in A$ denote the velocity vector field of a Kropina radial     geodesic $c(\tau):=\overline{\exp}_x(\tau y)$, $0\le \tau \le1$, emanating from $x$.
 Fix any instant $\tau \in (0, 1]$. Take any vector $V$ in the tangent space of
   \[ S_{A_x}(\tau r):=\{ y\in  A_x : F(x, y)=\tau r \}.\]
 
 Then, at the time  location $\tau $ of our Kropina radial geodesic, we have the orthogonality relation:
 \[g_{T}((\exp_{x*})_{\tau y}V, T)=0.\]
 \end{lemma}
 
{
The proof of the above Lemma is a particular case of  Remark 3.20 in \cite{JS}.
}

 It is a natural question to ask if a Kropina geodesic is an absolute minimum.
 In this section, we give an affirmative answer.
 
 Before proving it, we recall a fundamental property (Lemma 2.6 in \cite{Y}) of Finsler metrics
 \begin{eqnarray*}
 g_y(y, w)\le F(x, y)F(x, w)
 \end{eqnarray*}
 for any $y, w \in A_x$ (see also \cite{JS2}, Remark 2.9).
 Equality holds if and only if $w=ky$ for some $k>0$.
 
 Since we have
 \begin{eqnarray*}
 g_y(y, w)=(g_y)_{ij}y^iw^j=\frac{1}{2}\frac{\partial F^2(x, y)}{\partial y^j}w^j=F(x,y) F_{y^j}(x, y)w^j,
 \end{eqnarray*}
 we obtain
 \begin{eqnarray*}
  F_{y^j}(x, y)w^j\le F(x, w).
 \end{eqnarray*}

 We show the following Theorem.\\
 
 \begin{theorem}
 Let $(M, F: A \longrightarrow \R^+)$ be a Kropina space.
 Let $p\in M$ and $v\in S_{A_p}$.
 Suppose that $\exp_p$ is defined on the $B_{A_p}(r+\epsilon)\subset A_p$ and is diffeomorphism from $B_{A_p}(r+\epsilon)$ onto its image, where $r$ and $\epsilon$ are positive and
 \[ B_{A_p}(r+\epsilon):=\{y\in A_p | F(p, y)<r+\epsilon\}.\]
 
 Then, each radial Kropina geodesic $\overline{\exp}_p(tv)$, $0\le t \le r$, will minimize distance (which is $r$) among all piecewise smooth $F$-admissible  curves in $M$ that share its endpoints.
 \end{theorem}

\begin{proof}
	We will prove this theorem by using ${\exp}_p(tv)$, but the result can be extended to the case of $\overline{\exp}_p(tv)$ as well. Indeed,  notice that the curves emanating from $p$ in the direction tangent to the
	indicatrix are not $F$-admissible curves because its tangent vectors do not belong to $A_p$ (see Definition 4.7 and formula (2.2)).

	
	First, we prove that if $c(u)$, $0\le u \le 1$, is a piecewise smooth $F$-admissible  curve with endpoints $c(0)=p$ and $c(1)=\exp_p(rv)$, and if $c(u)$,  $0\le u \le 1$, lies inside
	$\overline{\exp}_p(B_{A_p}(r)\cup S_{A_p}(r)\cup\{0\})$, then its arclength is at least $r$.

	We must notice that the arclength of the radial Kropina geodesic  $\overline{\exp}_p(tv)$, $0\le t \le r$, is $r$.
	
	Since  $c(u)$, $0\le u \le 1$, lies inside $\overline{\exp}_p( B_{A_p}(r)\cup S_{A_p}(r)\cup \{0\})$, we can express it as
	\begin{eqnarray}\label{3.19}
	c(u)=\overline{\exp}_p(t(u)v(u)),
	\end{eqnarray}
	where $v(u)\in A_p$ and $F(p, v(u))=1$.
	Taking into account that the end points of $c$ are $c(0)=p$ and $c(1)=\overline{\exp}_p(rv)=\exp_p(rv)$, it must follow that
	\begin{eqnarray*}
		t(0)=0, \hspace{0.1in}  t(1)=r,   \hspace{0.1in} v(1)=v.
	\end{eqnarray*}

	We consider the following variation of the Kropina geodesic $\overline{\exp}_p(tv)$:
	\begin{eqnarray*}
		H(u, t):=\overline{\exp}_p(tv(u)),
	\end{eqnarray*}
	where $0\le t \le r$, $0\le u \le 1$, and put
	\begin{eqnarray}
	T_{H(u,t)}:=H_* \frac{\partial}{\partial t}=\frac{\partial  H}{\partial t},  \label{3.20}\\
	U_{H(u, t)}:=H_* \frac{\partial}{\partial u}=\frac{\partial  H}{\partial u}.  \label{3.21}
	\end{eqnarray}
	We can regard as $T_{H(u,t)}\in \pi_A^*(TM)_{(H, T)}$ and $ U_{H(u, t)}\in \pi_A^*(TM)_{(H,T)}$.
	
	For each fixed $u$, $T_{H(u,t)}$ represents the velocity vector of the $F$-unit speed Kropina geodesic $\overline{\exp}_p(tv(u))$.
	
	On the other hand, $U$ is represented by
	\begin{eqnarray*}
		U_{H(u, t)}=\overline{\exp}_{p*}\bigg(t \frac{dv(u)}{du}\bigg),
	\end{eqnarray*}
	where $\frac{dv(u)}{du}$ is tangent to the indicatrix $S_{A_p}$ at $v(u)$.
	
	Then, by the Gauss lemma, we have
	\begin{eqnarray*}
		g_{T}(U, T)=0.
	\end{eqnarray*}

	Returning to the curve $c$, we can rewrite (\ref{3.19})  as
	\begin{eqnarray*}
		c(u)=H(u, t(u)),
	\end{eqnarray*}
	and hence we have
	\begin{eqnarray*}
		\frac{dc}{du}=\frac{\partial H}{\partial u}+\frac{\partial H}{\partial t}\frac{dt}{du}.
	\end{eqnarray*}
	Using (\ref{3.20}) and (\ref{3.21}), we obtain
	\begin{eqnarray*}
		\frac{dc}{du}= U+T\frac{dt}{du}.
	\end{eqnarray*}

	Since we suppose that the curve is a piecewise smooth  $F$-admissible  curve, $\frac{dc}{du}$ must belong to $A_{c(u)}$.
	Applying the fundamental property mentioned above, we have
	\begin{eqnarray*}
		F(c, \frac{dc}{du})=F(c, U+T\frac{dt}{du})   \ge F_{T^i}(c, T)\big(U^i+ T^i\frac{dt}{du}\big),
	\end{eqnarray*}
	that is,
	\begin{eqnarray*}
		F(c, \frac{dc}{du})\ge F_{T^i}(c, T)U^i+ F_{T^i}(c, T)T^i\frac{dt}{du}.
	\end{eqnarray*}
	The first term in the right-hand side is equal to $g_{T}(U, T)=0$ from Gauss Lemma.
	The second term in the right-hand side is equal to $g_{T}(T, T)\frac{dt}{du}=\frac{dt}{du}$ since the radial Kropina  geodesic $\overline{\exp}(tv)$ has a constant $F$-unit speed  $g_{(c, T)}(T, T)=1$.
	Hence, we have the inequality
	\begin{eqnarray*}
		F(c, \frac{dc}{du}) \ge \frac{dt}{du}.
	\end{eqnarray*}
	Therefore, we obtain
	\begin{eqnarray*}
		L(c)=\int_0^1F(c, \frac{dc}{du}) du\ge\int_0^1 \frac{dt}{du}du =t(1)-t(0)=r.
	\end{eqnarray*}
	Thus, we prove that if $c(u)$, $0\le u \le 1$, is a piecewise smooth  $F$-admissible  curve with endpoints $c(0)=p$ and $c(1)=\exp_p(rv)$, and if $c$ lies inside
	$\overline{\exp}_p(B_{A_p}(r)\cup S_{A_p}(r)\cup\{0\})$, then its arclength is at least $r$.
	
	Next, we consider the case that the curve $c(u)$ does not lies necessarily inside    $\overline{\exp}_p(B_{A_p}(r)\cup S_{A_p}(r)\cup\{0\})$.
	
	{
	Let us remark that, since the indicatrix $S_{A_p}(1)$ of the Kropina space $(M,F)$ is obtained by $W$-translating of the unit sphere of the Riemannian space $(M, h)$, the set $S_{A_p}(1)$ is passing through the origin of $T_pM$, and hence  $S_{A_p}(r)$ also has the same property.
		Thus,   using the assumption that     $\exp_p$  is diffeomorphism from $B_{A_p}(r+\epsilon)$ onto its image,          the image   $\overline{\exp}_p( S_{A_p}(r)\cup\{0\})$   of the boundary $S_{A_p}(r)\cup\{0\}$ of  $B_{A_p}(r)$   by $\overline{\exp}_p$ is  closed. 
		Hence if the curve $c(u)$ does not lies necessarily inside    $\overline{\exp}_p(B_{A_p}(r)\cup S_{A_p}(r)\cup\{0\})$,  it necessarily intersects the set     $\overline{\exp}_p(S_{A_p}(r))$.
		Put $c(u_0)\in \overline{\exp}_p(S_{A_p}(r))$.
		
	}

	Then, from the argument above the region $\{c(u)  |  0\le u \le u_0\}$ already has arclength at least $r$.
	
	Thus we conclude that   each radial Kropina geodesic $\overline{\exp}_p(tv)$, $0\le t \le r$, will minimize distance (which is $r$) among all piecewise smooth  $F$-admissible curves in $M$ that share its endpoints.
	
	$\qedd$
	
\end{proof}

 \subsection{Long geodesics}

We have (see \cite{JS2}, Proposition 3.22 for a more general case)

 \begin{theorem}\label{long geod}
Let $p\in M$ be a point on a connected strong Kropina manifold $(M,F)$ with navigation data $(h,W)$, where $h$ is a complete Riemannian metric. If we  assume that $\exp_p$ is defined on the whole $A_p$, then for any $q\in \exp_p(A_p)$, there exists a minimizing unit speed geodesic $\exp_p(ty)$ from $p$ to $q$, for some direction $y\in A_p$.
 \end{theorem}

 For the proof see the proof of the more general case \cite{JS2}, Proposition 3.22. Also compare with the proof of the Proposition 7 in \cite{R}. 
 
 \begin{remark}
 Notice that, under assumptions in Theorem \ref{long geod}, the existence of an $F$-geodesic from $p$ to $q$ does not imply in general that there exists an $F$-geodesic from $q$ to $p$ as in the case of a classical Finsler space.
 \end{remark}
 
\begin{remark} 
One could be tempted to formulate some kind of Hopf-Rinow Theorem at this stage. Although this is possible,  since the separation function $d_F$ is not an authentic Finslerian distance function, we prefer not to do this, but to postpone it for few sections.  
\end{remark}

 
 \section{The geodesic connectedness of Kropina manifolds}
 \subsection{The Kropina geodesics of the Euclidean space}\label{subsec: Euclidean space}
We start with a simple example.

  Let $(x^i)$ be a global coordinate system of the Euclidean space  $(\E^n, \langle \cdot, \cdot \rangle)$.
  We consider the navigation data $(h,W)$ on $\E^n$, where $h=(\delta_{ij})$ is the canonical Euclidean metric, and $W=C^i\frac{\partial}{\partial x^i}$ be a unit Killing vector field on  $(\E^n, h)$, that is $C^i$ are constants and  $\sum_{i=1}^n(C^i)^2=1$, i.e. $||W||=1$, where $||\cdot ||$ is the usual Euclidean norm. Since we are using the canonical Euclidean metric, $C_i=\delta_{ij}C^j=C^i$, so we can freely lower the indices.
  
  The flow $\varphi_t:\E^n\to \E^n$ of $W$ is given by
  $\varphi_t(p)=p+Wt$, for $W=(C^1,\dots,C^n)$ and
  any $p=(p^1,\dots,p^n)$, $t\in \R$.

  We will consider  the strong Kropina space $(\E^n, F)$ induced by the navigation data $(h,W)$, that is
  \[
  F(p,y)=\frac{||y||^2}{2\langle W,y\rangle}=\frac{\alpha^2}{\beta},
  \]
  where $h=\langle \cdot,\cdot \rangle$ is the usual Euclidean inner product.
  
  Since an $h$-geodesic $\rho:\R\to \E^n$ emanating from a point
$\rho(0)=p=(p^i)\in \E^n$ with initial direction $\dot{\rho}(0)=a$ is expressed by
  \begin{eqnarray}
     \rho(t)=p+at,
  \end{eqnarray}
  where $a=(a^1,\dots,a^n)$, with $h$-arclength condition  $\sum_{i=1}^n(a^i)^2=1$,
  we obtain the $F$-geodesic of $(\E^n, F)$
  \begin{eqnarray}
  \mathcal P(t)=\varphi(t,\rho(t))=p+(a+W)t.
  \end{eqnarray}
 The non-degeneracy condition $\dot{\mathcal P}(0)\neq 0$ gives $a\ne -W$. One can easily see that indeed, $|\dot{\rho}(t)|_h=1$ if and only if $F(\mathcal P(t),\dot{\mathcal P}(t))=1$.

  \begin{figure}[H]
  	\begin{picture}(200,180)
  	\put(200,10){\vector(0,1){150}}
  	\put(150,60){\vector(1,0){150}}
 
  \put(219, 109){\circle*{2}}
  	\put(210, 100){${\bf p}$}
  		\put(250, 115){${\bf W_p}$}
  		\put(250, 0){${\bf W_p^\perp}$}
 
  	\put(219,109){\vector(3,2){30}}
  	\put(139,109){\circle*{2}}
  	\put(139,109){\vector(3,2){30}}
  	\put(126,109){${\bf q_2}$}
  	\put(290, 150){$\mathcal D_p^+$}
                  \put(170, 150){$\mathcal D_p^-$}
                  \put(279,68.8){\circle*{2}}
                  \put(284,70){${\bf q_1}$}
                  \put(219,109){\vector(3,-2){60}}
                  \put(249,20){\circle*{2}}
                  \put(250,20){${\bf q_3}$}
                  \put(139,109){\vector(1,0){80}}
  	\put(172,10){\line(-1, 3){40}}
  	\put(143, 134){${\bf W_{q_2}}$}
  	\put(160, 0){${\bf W_{q_2}^\perp}$}
  	\put(252,10){\line(-1, 3){50}}
  	\thicklines
  	\put(252,10){\line(-1, 3){50}}
  		\put(172,10){\line(-1, 3){40}}
  	\end{picture}
  	\caption{The Kropina two-dimensional Euclidean space}\label{fig: Euclidean plane}
  \end{figure}
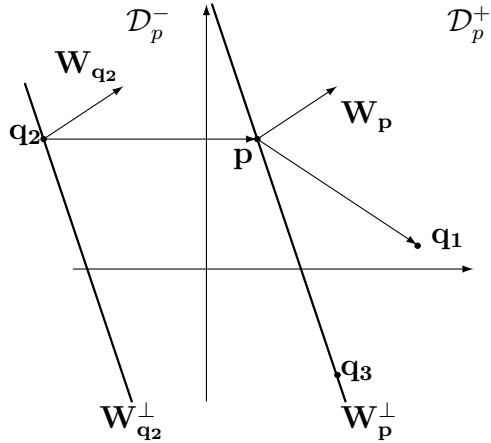%

   First thing to remark about this construction is that, given two unit vectors $a$ and $W$ in $\E^n$, $a\neq -W$, then, taking into account that
   $\dot{\mathcal P}(t)=a+W$ for any $t$, we have
 \[
 \langle\dot{\mathcal P},W\rangle=|\dot{\mathcal P}|\cos\angle ( \dot{\mathcal P},W) >0,
   \]
   that is $\angle(\dot{\mathcal P},W)<\frac{\pi}{2}$.
   Indeed, an elementary computation shows
   \[
    \langle\dot{\mathcal P},W\rangle=\langle a+W,W\rangle=\langle a,W\rangle+1=\cos\angle(a,W)+1>0
   \]
   and
   \[||\dot{\mathcal P}(t)||^2=||a||^2+2\langle a, W \rangle+||W||^2=2+2\langle a, W \rangle>0.\]

   We can summarize the $F$-geodesics behavior as follows (see Figure \ref{fig: Euclidean plane}).
   \begin{enumerate}
   	\item  The strong Kropina space $(\E^n,F)$ is {\it forward geodesically complete}, in the sense that any $F$-geodesic $\gamma:[a,b]\to M$ can be extended to $\gamma:[a,\infty)\to M$.
   	\item  There exists a Kropina geodesic from a given point $p$ to an arbitrary point $q_1\in \E^n$, $p\neq q_1$ if and only if
   	$q_1\in\mathcal D_{p}^+= \{q\in \E^n\ :\ \langle q-p,W \rangle >0\}$, i. e. $q_1$ is in the open half space of $\E^n$, determined by the hyperplane
   	$W^\perp_p:=\{ q\in \E^n\ :\ \langle q-p,W \rangle =0\}$ through $p$, that contains $W$. Hence $d_p:=d_F(p,\cdot):\mathcal D_{p}^+\to [0,\infty)$ is a well defined Lipschitz continuous function.
   	\item  If $q_2\in\mathcal D_{p}^-= \{q\in \E^n\ :\ \langle q-p,W \rangle <0\}$, by a similar construction as above, we can join $q_2$ to $p$ by an $F$-geodesic, but we cannot join $p$ to $q_2$.
   	\item If $ q_3\in W^\perp_p$, then it
                is not possible to join $p$ to $q_3$ nor $q_3$ to $p$ by an $F$-geodesic. In other words, for any $p\in \E^n$ we can summarize:
   	\begin{enumerate}
   		\item $\mathcal D_p^-\cap \mathcal D_p^+=\emptyset$;
   		\item $\mathcal D_p^-\cup \mathcal D_p^+=\E^n\setminus W_p^\perp$.
   	\end{enumerate}
   \end{enumerate}

 As it can be seen from this example, even though the Riemannian metric and the unit Killing vector field $W$ are both complete, it doesn't mean that we can join any two points by an $F$-geodesic.

\subsection{Geodesically connectedness}
The following definition is now natural.

\begin{definition}
Let $(M,F)$ be a strong Kropina manifold. A subset $D\subset M$ of $M$ is called {\it $F$-admissible pathwise connected} 
{({\it resp. $F$-geodesically connected})},
 or simply {\it $F$-connected}, if for any pair of points $p,q\in D$, there exists an $F$-admissible piecewise $C^\infty$-curve 
{({\it $F$-geodesic})}
joining $p$ and $q$. 
\end{definition}
 
 Observe that, due to the nature of Kropina merics, in the case of connectedness we may actually have a geodesic from $p$ to $q$ or one from $q$ to $p$, but not neccesarily both. Such kind of connected space could be called {\it weak geodesically connected} (this definition was suggested by the referee).

 The Hopf-Rinow Theorem for classical Finsler manifolds says that if $(M,F)$ is geodesically complete, then there exists an $F$-geodesic between any two points, that is $(M,F)$ is geodesically connected. Obviously this is not true any more for Kropina metrics as the example above shows.

This leads us to the fundamental question: {\it are there any geodesically connected Kropina spaces? }
 
 In the following we will study the characterization of such spaces.
 
 \begin{proposition}\label{prop: fix point characterization}
 Let $(M,F)$ be a strong Kropina space with the navigation data $(h,W)$, 
{where $h$, and hence $W$, are complete,}
and let $p,q\in M$ be two different points on $M$.
 Then, the followings are equivalent
 \begin{enumerate}[(i)]
   \item There exists a 
{minimizing } Kropina geodesic from $p$ to $q$.
   \item The mapping $\delta:[0,\infty)\to (0,\infty)$, $\delta(\tau):=d^h_p\circ \varphi_{-\tau}(q)$ has a fixed point, where $d^h_p:M\to [0,\infty)$, $d^h_p(q)=d^h(p,q)$ is the Riemannian distance function from $p$, and $\varphi_t(q)$ is the flow of $W$ through $q$ parametrized such that
   $\varphi_0(q)=q$.
  \end{enumerate}
 \end{proposition}
 \begin{proof}
 
 (i)$\Rightarrow$ (ii)

\bigskip

\setlength{\unitlength}{1cm}
  \begin{figure}[h]
    \begin{center}
\begin{picture}(6, 4)

\qbezier(-1,4.5)(3,3)(7,4.5)
\put(1.5,3.85){\circle*{0.1}}
\put(4.3,3.85){\circle*{0.1}}
\put(2,0.7){\circle*{0.1}}
\qbezier(2,0.7)(1.2,3)(1.5,3.85)
\qbezier(2,0.7)(3,3)(4.3,3.85)
\put(1.5,4.1){$q^-$}
\put(4.3,4.1){$q$}
\put(2,0.2){$p$}
\put(-1,4.1){$\varphi_q$}
\put(3,3.75){\vector(1,0){0.1}}
\put(3.35,3){\vector(1,1){0.1}}
\put(1.43,3){\vector(0,1){0.1}}
 \put(3.7,3){$\mathcal P$}
 \put(1,3){$\rho$}
 \end{picture}
\caption{Constructing an $F$-geodesic from $p$ to $q$ via $q^{-1}=\varphi_{-l}(q)=\rho(l)$.}\label{fig: Constructing an $F$-geodesic}
    \end{center}
  \end{figure}
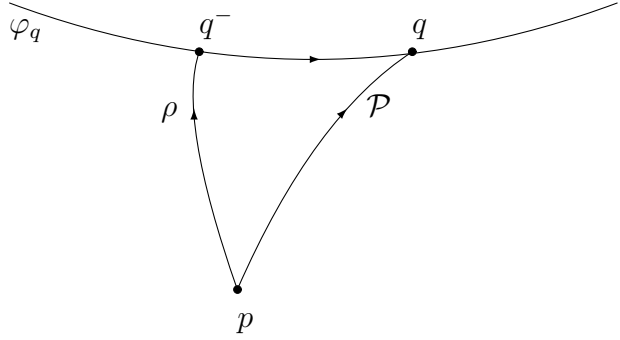

 Let points $p$ and $q$ be as in hypothesis and assume there exists a minimal $F$-unit speed Kropina geodesic $\mathcal P:[0,l]\to M$ from $p$ to $q$, that is
 $\mathcal P(0)=p$, $\mathcal P(l)=q.$
 
 Then, from Theorem \ref{Theorem 3.9} it follows that there exists an $h$-unit speed Riemannian geodesic $\rho:[0,l]\to M$, from $p$, that is $\rho(0)=\mathcal P(0)=p$ given by $\rho(t)=\varphi_{-t}(\mathcal P(t))$. Since we can use the same parameter on both geodesics $\rho$ and
 $\mathcal P$ (see Remark \ref{rem: equiv}), we can easily see that $q=\mathcal P(l)=\varphi_l(\rho(l))$ implies the existence of a point $q^-:=\varphi_{-l}(q)$ and hence
 \[
 \rho(l)=q^-=\varphi_{-l}(\mathcal P(l))=\varphi_{-l}( q),
 \]
 that is $\delta(l)=l$ and we have obtained a fixed point for $\delta$.
 
 (ii)$\Rightarrow$ (i)
 
 Conversely, we assume that there exists $\tau_0>0$ such that $\delta(\tau_0)=\tau_0$. Observe that $\delta(0)=d^h(p,q)>0$.
 
 In other words, there exists a point $q^-:=\varphi_{-\tau_0}(q)$ on the integral curve of $W$ through $q$. We consider the $h$-unit length geodesic $\rho:[0,\tau_0]\to M$ from $p=\rho(0)$ to $q^-=\rho(\tau_0)$.
 
 Then $\mathcal{P}(t):=\varphi_t(\rho(t))$ is an $F$-unit speed Kropina geodesic from $p$ to $q$ (see Figure \ref{fig: Constructing an $F$-geodesic}). Observe that $\dot{\mathcal{P}}(0)=\dot{\rho}(0)+W_p$.
 
 The only thing we have to prove is the $F$-admissibility of this $\mathcal{P}$, that is $\dot{\rho}(0)\neq -W_p$. 
 It is enough to verify this at $t=0$ (see Lemma \ref{lem: non-degenerate curve}), i.e. $h(\dot{\mathcal{P}}(0),W_p)>0$.

 Indeed, if we assume $\dot{\rho}(0)= -W_p$, then it follows that $\rho$ is an integral curve of $-W$, or, equivalently, an integral curve of $W$, in other words, $p$ must be on the integral curve of $W$ through $q^-$.
By construction we have $\varphi_{\tau_0}(q^-)=p$.
On the other hand, from  $q^-:=\varphi_{-\tau_0}(q)$ we have $\varphi_{\tau_0}(q^-)=q$.
It follows that $p=q$.

This is not possible by hypothesis, so we have got a contradiction, hence $\mathcal{P}$ is $F$-admissible.
 
The fact that $\mathcal{P}$ is $F$-unit speed geodesic follows from construction.
 $\qedd$
 \end{proof}
 
 We observe that there are points on $M$ that cannot be joined with $F$-geodesics.

 Here is our main result of this section
 
 \begin{theorem}\label{thm: quasi-regular complete}
 Let $(M,F)$ be a forward complete strong Kropina space with the navigation data $(h,W)$, where {$h$ is complete}
 and $W$ is a
quasi-regular unit Killing vector field.
 
 Then, we can join a point $p$ to any point $q$ by a minimizing $F$-unit speed Kropina geodesic.

 \end{theorem}
 \begin{proof}
 	
 By hypothesis, all integral curves of $W$ are simple closed geodesics, even though they might not be of same length (see Definition \ref{def: quasi-regular}). In any case, the mapping $\delta$ defined in Proposition \ref{prop: fix point characterization} is defined on the compact set $[0,2\pi c]$, where $c:=\frac{1}{2\pi}\mathcal L_h(\varphi_t(q))$, if $l_0=d^h(p,q)$ is less than the $h$-length of the closed integral line of $W$ through $q$, or on $[0,2n\pi c]$ for some positive integer $n$ chosen such that $l_0$ is less than the $h$-length of $\varphi([0,2n\pi c],q)$. This is always possible by passing to the universal cover of $M$, for example.
 
 Now, observe that in our hypotheses, $(M,h)$ is complete Riemannian manifold.
 
 For the sake of simplicity we assume $n=1$,  that is $l_0<2\pi c$. From $l_0=\delta(0)=\delta(2\pi c)$ it follows $0<\delta(0)=\delta(2\pi c)<2\pi c$. Since $\delta$ is continuous, it must exists a $\tau_0$ such that $\delta(\tau_0)=\tau_0$ , hence there exists an $F$-geodesic from $p$ to $q$ due to Proposition \ref{prop: fix point characterization} (see Figure \ref{fig: The fixed point tau zero}).

 \setlength{\unitlength}{1cm}
 \begin{figure}[h]
 	\begin{center}
 		\begin{picture}(6,5)
 		\put(-3.5,1){\vector(1,0){6}}
        \put(-3,0.5){\vector(0,1){4.5}}
        \thinlines
         \put(-3,4.5){\line(1,0){3.5}}
            \put(0.5,1){\line(0,1){3.5}}
            \put(-3,1){\line(1,1){3.5}}
            \put(-3,2.5){\line(1,0){3.5}}
             \put(-4,4.5){$2\pi c$}
             \put(0.5,0.5){$2\pi c$}
              \put(2.5,0.5){$\tau$}
              \put(-4.6,2.5){$l_0=\delta(0)$}
              \put(0.7,2.5){$\delta(2\pi)$}
              \put(-0.85,1){\line(0,1){2.15}}
              \put(-0.85,0.5){$\tau_0$}
              \put(-1,4){$\delta$}
              \thicklines
              \qbezier(-3,2.5)(-1.3,5)(-0.6,2.5)
              \qbezier(-0.6,2.5)(0,0.6)(0.5,2.5)
 	\end{picture}
 	\caption{The fixed point $\tau_0$ of $\delta$.}\label{fig: The fixed point tau zero}
 \end{center}
\end{figure}
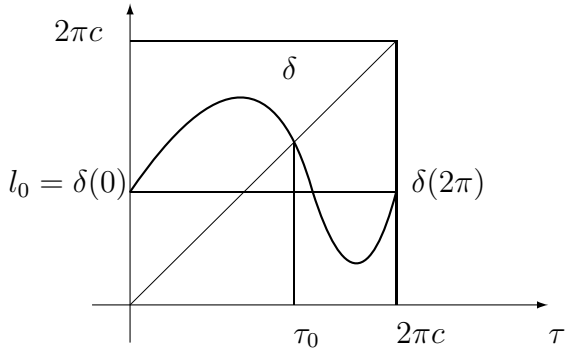

 $\qedd$
 
 \end{proof}

\begin{corollary}
If $(M,F)$ is a forward complete strong Kropina space with navigation data $(h,W)$, where
{$h$ is complete}
 and $W$ is a 
 quasi-regular unit Killing vector field, then the separation function $d_F$ is a non-symmetric distance function on $M$, and hence $(M,d_F)$ is a non-symmetric metric space. 
\end{corollary}

{
Indeed, since for any points $p$, $q\in M$ there is a minimizing $F$-unit speed geodesic from $p$ to $q$ from Theorem \ref{thm: quasi-regular complete}, $d_F(p, q)$ can be defined as usual.
The {\it positivity}, {\it positive definiteness} and {\it triangle inequality} of $d_F$ naturally follow. Clearly, $d_F$ doesn't need to be symmetric.

}

Under the hypotheses of the Corollary above, since $(M,d_F)$ is a metric space, we can define the notions of convergent sequences, forward (backward) Cauchy sequences of points on $M$ as usual (see for instance \cite{BCS}). Such a metric space $(M,d_F)$ will be called forward (backward) complete if every forward (backward) Cauchy sequence is convergent, respectively.

A Hopf-Rinow Theorem for strong Kropina spaces can be now formulated (see \cite{JS2}, Theorem 3.22 for a more general case).

 \begin{theorem}\label{Finslerian HR theorem}
 	Let $(M,F:A\to \R^+)$ be a connected strong Kropina manifold with navigation data $(h,W)$, where $W$ is a quasi-regular unit Killing vector field. Then the followings are equivalent
 	\begin{enumerate}[(i)]
 		\item The non-symmetric metric space $(M,d_F)$ is forward complete (as metric space).                                          
 		\item The Kropina manifold $(M,F)$ is forward geodesically complete.
 		\item For any $p\in M$ the exponential map $\exp_p$ is defined on the whole $A_p\subset T_pM$.        
 	       \item Any closed and forward bounded subset in  $(M,d_F)$ is compact.
 	\end{enumerate}
 	Moreover, if one of the above holds, then for any pair of points in $M$ can be joined by a minimizing geodesic.
 \end{theorem}
 	\begin{proof}
 		First, suppose that (i) holds.
 		Let $\mathcal P(t)$, $a\leqq t <b$, be a maximally forward extended $F$-unit speed geodesic.
 		We will show that $b=\infty$.
 		So, we suppose that $b\ne \infty$.
 		Then we have an increasing sequence $\{t_i\}$ in $[a, b)$ converging to $b$.
 		So, there exists an integer $N$ such that
 		\[    N<i<j  \Longrightarrow   t_j-t_i<\epsilon\]
 		for any $\epsilon>0$.
 		
 		Since $\mathcal P(t)$ is an $F$-unit speed geodesic, we have
 		\[ i<j  \Longrightarrow  d_F(\mathcal P(t_i), \mathcal P(t_j)) <t_j-t_i.\]
 		
 		From the above two properties, it follows that there exists a  positive integer $N$ such that for any $\epsilon>0$
 		\[    N<i<j  \Longrightarrow                    d_F(\mathcal P(t_i), \mathcal P(t_j))<\epsilon,\]
 		that is, $\{\mathcal P(t_i)\}$ is a forward Cauchy sequence.
 		Then, from (i), it follows that $\{\mathcal P(t_i)\}$ converges to $q\in M$.
 		Define $\mathcal P(b)=q$.
 		From the theory of ordinary differential equations, it follows that $\mathcal P(t)$ is defined on a neighborhood of $t=b$.
 		This is a contradiction with the properties of $b$.
 		Hence it follows that $b=\infty$, that is, (ii) follows.
 		
 		Secondly, suppose that (ii) holds.
 		Then (iii)  automatically follows.
 		
 		Thirdly, suppose that (iii) holds.
 		Let $\mathcal{N}$ be a closed and forward bounded subset of $(M,d_F)$.
 		Since $\mathcal{N}$ is a forward bounded subset, there exists a point $p\in M$ and $r>0$ such that
 		\[     \mathcal{N} \subset                  \mathcal{B}_F^+(p,r).\]
 		For each $q\in \mathcal{N}$, there exists  $v_q\in A_p$ such that $\exp_p{(tv_q)}$, $0<t\leqq 1$, is a geodesic from $p$ to $q$.
 		The collection of all $v_q$ is a subset $N$ of $A_p$.
 		This subset is bounded because $F(p, v_p)=d_F(p, q)<r$, that is, $N$ is contained in the compact set  $B_p^F(r)\cup S_p(r)$. 
 		Hence $\mathcal{N}=\exp_p{N}$ is a closed set contained in the compact set 
 		$\exp_p{\big( B_p^F(r)\cup S_p(r)}$). 
 		Therefore,  $\mathcal{N}$ is a compact set and $(iv)$ follows.
 		
 		Lastly, suppose that (iv) holds.
 		Let $\{p_i\}$ be a forward Cauchy sequence in $(M,d_F)$.
 		Then there exists a positive integer $N$ such that for any $\epsilon>0$
 		\[N<i<j  \Longrightarrow                    d_F(p_i, p_j)<\epsilon.\]
 		So, for any  $i(>N)$ the inequality
 		\[ d_F(p_1, p_i)\leqq d_F(p_1, p_{N+1})+d_F(p_{N+1}, p_i)< d_F(p_1, p_{N+1})+\epsilon\]
 		holds, that is, the Cauchy sequence $\{p_i\}$ is forward bounded.
 		This closure, denoted by  $\mathcal{A}$, is also forward bounded and of course a closed set.
 		Hence, from the assumption $(iv)$  $\mathcal{A}$ is compact.
 		
 		We will show that the sequence $\{p_i\}$ contains a convergent subsequence.
 		To do so, we suppose that it has no convergent subsequence.
 		Then for each $p\in \mathcal{A}$ there exist a point $\overline{p}\in \mathcal{D}^+$ and a small forward metric ball  $\mathcal{B}^+_F(\overline{p},r_p)\ni p$ which contains no point of $\{p_i\}$ except $p$.
 		These forward balls are open set and cover $\mathcal{A}$.
 		Since $\mathcal{A}$ is compact, the finite number of them covers $\mathcal{A}$.
 		Each  of the finite sub-covering  can contain at most one point of the sequence $\{p_i\}$.
 		This means that the point set of the sequence is finite.
 		There is at least one convergent subsequence, and this is a contradiction.
 		Therefore, it follows that the sequence $\{p_i\}$ contains a convergent subsequence.
 		
 		Let denote the above convergent  subsequence as $\{p_\alpha\}$.
 		Say it converges to some $q\in \mathcal{A}$.
 		
 		We will show that the sequence $\{p_i\}$ converges to $q$.
 		Since the sequence $\{p_i\}$ is a forward Cauchy sequence, there exists a positive integer $N_1$ such that
 		\[    N_1<i<j  \Longrightarrow                    d_F(p_i, p_j)<\frac{\epsilon}{2}.\]
 		Furthermore, since the subsequence $\{p_\alpha\}$ converges to $q$, there exists a positive integer $N_2$ such that
 		\[  N_2<\alpha  \Longrightarrow   d_F(p_\alpha, q)<\frac{\epsilon}{2}.\]
 		Choose an integer $N$ such that  $N>max\{N_1, N_2\}$, we get
 		\[ N<i  \Longrightarrow  d_F(p_i, q)<d_F(p_i, p_N)+d_F(p_N, q)< \frac{\epsilon}{2}+\frac{\epsilon}{2}=\epsilon.\]
 		Therefore the Cauchy sequence $\{p_i\}$ converges to $q$.
 		$\qedd$
 		
 	\end{proof}

 \begin{remark}
 Let us  recall that the Poincar\'e conjecture (see \cite{KL}) states that every compact connected simply connected 3-dimensional manifold $M$
{is homeomorphic to the three-dimensional sphere $\Sph^3$.  Observe that since we are in dimension three,  $M$ is actually diffeomorphic to $\Sph^3$.  
	
	 Therefore $M$}
 admits a Riemannian metric $h$ and a unit quasi-regular Killing vector field $W$ (see \cite{BN1} for the construction of these). 
By using
{Theorem \ref{Finslerian HR theorem} to $M$ it follows that}
every compact connected simply connected 3-dimensional manifold $M$, i.e.  a manifold homeomorphic to $\Sph^3$,
admits a strong Kropina metric which is  geodesically connected (see also \cite{YS1}).
 
 \end{remark}

 	\subsection{Strong Kropina metrics on Riemannian homogeneous spaces}
 	We end this section with another example. 
 	
 	It is known that even dimensional Riemannian homogeneous spaces $(M,h)$ of constant positive curvature are isometric either to an
 	even dimensional  sphere $\Sph^{2n}$, or to a real projective space, and in these cases, $M$ does not carry a quasi-regular Killing field of constant length.
 	
 	\begin{remark}
 		We recall that a homogeneous Riemannian manifold $M$ of dimension $m\geq 2$ has constant sectional curvature one if and only if
 		$M=\Sph^m\slash \Gamma$, where $ \Sph^m$ is the round sphere of unit constant sectional curvature, and $\Gamma$ is some discrete group of Clifford-Wolf translations.
 		
 		It is also known that such a $\Gamma$ is isomorphic either to a cyclic group, or to a binary dihedral group, or to a binary polyhedral group (see \cite{W}). Clearly, for $\Gamma=\{I\}$ and $\Gamma=\{\pm I\}$ the space $M=\Sph^m\slash \Gamma$ is an Euclidean sphere and a real projective space, respectively (here $I$ is the antipodal map).
 	\end{remark}
 	
 	On the other hand, in the odd dimensional case, there are situations where we can construct geodesically connected strong Kropina manifolds.
 	
 	\begin{proposition}
 		Let $\Gamma$ be a cyclic group of Clifford-Wolf translations on an odd dimensional sphere $\Sph^{2n-1}$, $n\geq 2$, of constant sectional curvature, which is distinct from the groups $\Gamma=\{I\}$ and $\Gamma=\{\pm I\}$.
 		
 		Then, the homogeneous space $M=\Sph^{2n-1}\slash \Gamma$, admits geodesically connected strong Kropina manifolds.
 	\end{proposition}
 	
 	Indeed, under our hypotheses, the Riemannian homogeneous space $M=\Sph^{2n-1}\slash \Gamma$ admits a quasi-regular Killing field $W$ of constant length. For its effective construction see the proof of Proposition 13 in \cite{BN1}.
 	
 	Moreover, we have
 	
 	\begin{proposition}
 		Every homogeneous Riemannian space $M=\Sph^{4n-1}\slash \Gamma$ admits geodesically connected strong Kropina manifolds, where
 		$\Gamma$ is a discrete group of Clifford-Wolf translations of $\Sph^{4n-1}$, which is isomorphic either to a binary dihedral group, or to a binary polyhedral group.
 	\end{proposition}

 \section{Conjugate and cut points}
 
 \subsection{Conjugate points }

 \hspace{0.2in}  Let $(M, F=\alpha^2/\beta : A\to\R^+)$ be a globally defined forward complete strong Kropina space  with  navigation data   $(h, W)$.

 For any point $p\in M$,  we will consider the correspondence between the conjugate points of $p$ with respect to the Riemannian metric $h$ and the conjugate points of $p$ with respect to the Kropina metric $F$.

 Let $\mathcal P(t) : [0, l] \longrightarrow M$ denote an $F$-unit speed geodesic such that $\mathcal P(0)=p$ and $\mathcal P(l)=q_{_F}$.
 The curve $\rho(t):=\varphi_{-t}(\mathcal P(t))$ is the corresponding $h$-geodesic in $(M, h)$ (see Theorem   \ref{Theorem 3.9}).

 Suppose that $\rho(l)=q_h$ is conjugate to $p$ along $\rho$.
 Then, there exists a Jacobi field $J$ along $\rho$ such that $J(0)=J(l)=0$.
 There exists a geodesic variation $\rho_s : [0, l] \longrightarrow M$, $-\epsilon < s < \epsilon$, of $\rho=\rho_0$  with the variational vector field
 \[	\frac{\partial \rho_s}{\partial s}\bigg|_{s=0}=J,  \]
 (see for eg. Proposition 2.4 in \cite{Ca}).
 
 We will construct an $F$-admissible Kropina geodesic variation by using the above Riemannian geodesic variation.
 Since $\rho_s(t)$ is a geodesic, $|\dot{\rho}_s(t)|_h$ is constant along $\rho_s(t)$.
  So, we put $|\dot{\rho}_s(t)|_h=v(s)$. To reparametrize the $\rho_s(t)$ so that the composition with  $\varphi_t$ is a Kropina geodesic, we put $t=u/v(s)$ and
  \[	\widetilde{\rho}_s(u):=\rho_s(\frac{1}{v(s)}u)=\rho_s(t),\quad -\epsilon < s < \epsilon.
  \]
 Hence, we get $|\dot{\widetilde{\rho}}_s(u)|_h=1$.
 
 Furthermore, since we have $\dot{\mathcal P}(t) \ne 0$,  we get $\dot{\rho}(t)\ne -W(\rho(t))$.
 So,  we can suppose that for a sufficiently small $\epsilon >0$ the tangent vector $\dot{\widetilde{\rho}}_s(u)$, $-\epsilon <s<\epsilon$, is not $-W( \widetilde{\rho}_s(u))$.

 Therefore, we get  Kropina geodesics
 \[	\widetilde{\mathcal P}_s(u):=\varphi_u(\widetilde{\rho}_s(u))=\varphi_{v(s)t}(\rho_s(t))=:\mathcal P_s(t),\quad -\epsilon < s < \epsilon.  \]
 Since the geodesic $\widetilde{\mathcal P}_s(u)$ is $F$-unit constant speed, the geodesic $\mathcal P_s(t)$ must have the $F$-speed $v(s)$.
 
 We consider the $F$-geodesic variation $\mathcal P_s(t)$, $-\epsilon < s < \epsilon$, 
 of $\mathcal P(t)=\mathcal P_0(t)$.
 Denoting a Jacobi field along $\mathcal P(t)$ by $\mathcal{J}$, we have
 \begin{equation}\label{6.24}
 	\mathcal{J}(t):=\frac{\partial \mathcal P_s(t)}{\partial s}\bigg|_{s=0}
 	             =(v'(0) W_{\mathcal P})t+(d\varphi_t)(J),
 \end{equation}
 where $d\varphi_t$ denotes the differential of $\varphi_t$.

 Since $v(s)=\sqrt{h(\dot{\rho}_s(t), \dot{\rho}_s(t))}$, we get
 \begin{eqnarray*}
 	v'(s)=\frac{h(\dot{\rho}_s(t), \frac{\delta \dot{\rho}_s(t)}{\delta s})}{\sqrt{h(\dot{\rho}_s(t), \dot{\rho}_s(t))}}
 	     =\frac{h(\dot{\rho}_s(t), \frac{\delta}{\delta t}\frac{\partial \rho_s(t)}{\partial s})}{\sqrt{h(\dot{\rho}_s(t), \dot{\rho}_s(t))}}.
 \end{eqnarray*}
 Putting $s=0$, we get
 \[	v'(0)=\frac{h(\dot{\rho}(t), \frac{\delta}{\delta t}J(t))}{\sqrt{h(\dot{\rho}(t), \dot{\rho}(t))}}
 	     = \frac{ \frac{\delta}{\delta t}h(\dot{\rho}(t),J(t))}{\sqrt{h(\dot{\rho}(t), \dot{\rho}(t))}}.  \]
 Since the Jacobi field $J$ satisfies the conditions $J(0)=0=J(l)$, it follows that $h(\dot{\rho}(t), J(t))\equiv 0$ for all $t \in [0, l]$ (see for e.g. Corollary 3.7 in \cite{Ca}).
  Hence, we get $v'(0)=0$.
 Therefore, the equation (\ref{6.24}) reduces to
 \[	\mathcal{J} =(d\varphi_t)(J). \]
 Since $\varphi_t$ is an isometry, it follows that the Jacobian $d\varphi_t$ is nonsingular. Hence, $\mathcal{J}$ vanishes if and only if $J$ does. In particular, $\mathcal{J}$ is non-zero and $\mathcal{J}(0)=0=\mathcal{J}(l)$. Equivalently, $q_{_F}=\mathcal P(l)$ is conjugate to $p=\mathcal P(0)$ along $\mathcal P$ with respect to $F$.
 
 By the above argument, we proved the following :
 
 \vspace{0.05in}
  {\it Suppose that $\mathcal P : [0, l] \longrightarrow M$ is an $F$-unit speed geodesic with the associated Riemannian geodesic $\rho$. Then, $\mathcal P(l)$ is conjugate to $\mathcal P(0)$ with respect to $F$ whenever $\rho(l)$ is conjugate to $\rho(0)$ with respect to  $h$.}
 \vspace{0.05in}

 In a similar way to the above proof, we can show the converse  statement  :
 \vspace{0.05in}
 
 {\it Suppose that  $\rho : [0, l] \longrightarrow M$ is an $h$-unit speed geodesic with the associated  Kropina geodesic $\mathcal P$.
 Using the above result, it follows that $\rho(l)$ is conjugate to $\rho(0)$ with respect to $h$ whenever $\mathcal P(l)$ is conjugate to $\mathcal P(0)$ with respect to $F$.}
 \vspace{0.05in}
 
 Summarizing the above discussion, we get

 \begin{theorem}\label{Theorem 3.17}
 Let $(M, F=\alpha^2/\beta : A\to\R^+)$ be a globally defined strong Kropina space with navigation data  $(h, W)$.
 
 Suppose that there exists a one-parameter group $\{\varphi_t\}$ of isometries on $(M, h)$ which generates the unit Killing vector field $W$ and that  $\mathcal P : [0, l] \longrightarrow M$ is an $F$-unit speed geodesic of the Kropina space $(M, F)$.
 
 Then, $\mathcal P(l)$ is conjugate to $\mathcal P(0)$ along $\mathcal P$ with respect to $F$ if and only if $\rho(l)$ is conjugate to $\rho(0)$ along the $h$-Riemannian geodesic $\rho(t):=\varphi_{-t}( \mathcal P(t)) $.
 \end{theorem}


 \subsection{Cut points}\label{sec: Cut points}
 
Similarly with the classical Finsler manifolds, if a unit speed (non-constant) geodesic segment $\mathcal P:[0,l]\to M$ is maximal, as geodesic segment, then the point $q:=\mathcal P(l)$ is called a {\it cut point} of the point $p:=\mathcal P(0)$ along the $F$-geodesic $\mathcal P$ (see \cite{BCS} and \cite{ST} for classical case). 

The {\it cut locus} of $p$, denoted by $\mathcal C_p$ is the set of all cut points along all non-constant $F$-geodesic segments from $p$.

\begin{remark}\label{rem: classical cut locus prop}
	In the case of a classical Finsler manifold, the following basic properties of the cut locus are well-known (\cite{BCS}, \cite{ST}) 
	\begin{enumerate}
		\item $p\notin \mathcal C_p$.
		\item If a point $q\in M$ admits (at least) two geodesic segments from $p$ of equal length, then $q\in \mathcal{C}_p$. 
		\item If $q\in \mathcal{C}_p$ then one of the following situations happens
			\begin{enumerate}
				\item $q$ is conjugate to $p$, or
				\item there exist (at least) two geodesic segments from $p$ to $q$ of equal length.
			\end{enumerate}
	\end{enumerate}
	We point out that these properties are not true in the case of a generic Kropina manifold as will be seen in the examples below. 
\end{remark}

\begin{proposition}
	Let $(M,h)$ be a complete Riemannian manifold and let $(M,F:A\to \R^+)$ be a strong Kropina manifold with navigation data $(h,W)$, where $W$ is a quasi-regular Killing field. 
	
	Then any two points $p,q\in M$ can be joined by a globally $F$-length minimizing geodesic.

\end{proposition}
\begin{proof}
	The completeness of the Riemannian metric $h$ implies that, for any $x\in M$,  the $h$-exponential map $e_x$ is defined on whole $T_xM$ (Hopf-Rinow Theorem for Riemannian manifolds). 
	
	On the other hand, remark that $W$ is complete vector field on $M$.
	
	Using \eqref{F and h-exponential maps} it follows that, for each $x\in M$, the $F$-exponential map $\overline{\exp}_x$ is defined in whole $A_x\subset T_xM$ and hence $(M,F)$ is complete. Theorem  \ref{Finslerian HR theorem} implies that any two points on $M$ can be joined by a globally $F$-length minimizing geodesic as stated.
		$\qedd$
	\end{proof}
	
	\begin{remark}
		Taking into account that $W$ is a unit Killing field it results that 
		$\mathcal P(t)=\varphi_t(\rho(t))$, $t\in [0,l]$, is a global minimizer of $F$-path lengths from $p=\mathcal P(0)$ to $q=\mathcal P(l)$ if and only if the corresponding $h$-geodesic $\rho(t)$, $t\in [0,l]$, is a global minimizer of $h$-path lengths from $p=\rho(0)$ to $\hat{q}=\rho(l)$. 
	\end{remark}
	
	It follows

 \begin{proposition}
\label{cor: F-cut locus}
Let $(M,h)$ be a complete Riemannian manifold and let $(M,F:A\to \R^+)$ be a strong Kropina manifold with navigation data $(h,W)$, where $W$ is a
 unit Killing field. 
			
	Then the point $q=\mathcal P(l)$ is an $F$-cut point of $p=\mathcal P(0)$  along the $F$-geodesic segment $\mathcal P(t)$ if and only if $\hat{q}=\rho(l)$ is an $h$-cut point of $p=\rho(0)$  along the $h$-geodesic segment $\rho(t)$. 
\end{proposition}

 \section{Examples}
 
The results presented in the previous sections 
 can be used for constructing complete, geodesically connected strong Kropina spaces.
 
 \subsection{Odd dimensional spheres}
 
 \subsubsection{The action of $\Sph^1$ on $\Sph^{2n-1}$}
 
 Let $M=\Sph^{2n-1}$ be the round sphere in $\R^{2n}$ with the canonical induced metric $h$. We regard $\Sph^{2n-1}$ as the subset
 \[
 \Sph^{2n-1}=\{z=(z_1,z_2,\dots.z_n)\in \C^n\ :\ ||z||=1\}
 \]
 of the complex space $\C^n(=\R^{2n})$, $n\geq 2$, with the canonical Hermitian norm.
 
 We define the action of  $\Sph^1$ on $\Sph^{2n-1}$ as follows
 \begin{equation}\label{canonical action}
 \varphi:\Sph^1\times \Sph^{2n-1}\to \Sph^{2n-1},\quad \varphi(s,z)=(s^{a_1}z_1,s^{a_2}z_2,\dots, s^{a_n}z_n ),\quad s=e^{it},
 \end{equation}
 where $a_1\geq a_2\geq \dots \geq a_n$ are some (positive) real constants. It is easy to see that this is an isometric action of   $\Sph^1$ on $\Sph^{2n-1}$ with respect to $h $.

  The vector field generated by this action is
  \[
  W_z=i(a_1z_1,a_2z_2,\dots,a_nz_n).
  \]
  
 
 \begin{remark}
 In the case $\Sph^3\subset \C^2= \R^4$, the action $\varphi:\Sph^1\times \Sph^3\to \Sph^3$ can be written as
\begin{equation*}
\varphi_{t}(x)=\begin{pmatrix}
\cos a_1t & -\sin a_1t & 0 & 0 \\
\sin a_1t & \cos a_1t & 0 & 0 \\
 0 & 0 & \cos a_2t & -\sin a_2t \\
 0 & 0 & \sin a_2t & \cos a_2t
\end{pmatrix}
\begin{pmatrix}
x_1 \\ x_2\\ x_3\\ x_4
\end{pmatrix},
\end{equation*}
 where $z_1=x_1+ix_2$, $z_2=x_3+ix_4$, and $x=(x_1,x_2,x_3,x_4)\in \Sph^3$.

  In this case, we obtain
  \[
  W_x=\begin{pmatrix}
  0 & -a_1 & 0 & 0 \\
  a_1 &  0 & 0 & 0 \\
   0 & 0 & 0 & -a_2 \\
   0 & 0 & a_2 &  0
  \end{pmatrix}
  \begin{pmatrix}
  x_1 \\ x_2\\ x_3\\ x_4
  \end{pmatrix}.
  \]
 
 \end{remark}
 
 In the special case $a_1=a_2=\dots=a_n=1$ it follows $W$ is an $h$-unit Killing vector field on the round sphere $\Sph^{2n-1}$, and hence we obtain the navigation data $(h,W)$ that induces a strong Kropina metric on  $\Sph^{2n-1}$. We will call in the following $(h,W)$  and the induced strong Kropina metric the {\it canonical navigation data} and the {\it canonical Kropina metric} of $\Sph^{2n-1}$, respectively.
 
 
 \subsubsection{The geodesics on the round sphere $(\Sph^{2n-1},h)$}
 
 On the round sphere  $(\Sph^{2n-1},h)$,
 for a point $z=(z_1,z_2,\dots.z_n)\in \Sph^{2n-1}$ and a tangent vector $v=(v_1,v_2,\dots,v_n)\in T_z\Sph^{2n-1}$,
 $| |z| |=| |v| |=1$,
 the $h$-geodesic $\rho:[0,2\pi]\to \Sph^{2n-1}$ with initial conditions $(z,v)$ is the great circle
 \[
 \rho(t)=z\cos t+v\sin t.
 \]
 
 Clearly, the $h$-geodesics are $2\pi$-periodic and minimizing from $p$ to its antipodal point.
 
 
 \subsubsection{The geodesics of $(\Sph^{2n-1},F)$}\label{Subsubsection 8.1.3}
 
 By using the expressions given already for the $h$-geodesics and the flow $\varphi$ we obtain the $F$-geodesics of the canonical Kropina metric on
 $\Sph^{2n-1}$ by means of Theorem \ref{Theorem 3.9}.
 
 Indeed,
 \begin{equation}\label{F-geodesic on the sphere}
\mathcal P(t)=\varphi_{t}(\rho(t))=s(t)(z\cos t+v\sin t),
 \end{equation}
 where $s(t)=\cos t+i \sin t\in \Sph^1$, and $(z,v)$ are the initial conditions of the $h$-geodesic $\rho$.
 
 \begin{remark}
 In the case of $\Sph^3$ we get
 \[
 \mathcal P(t)=
 \begin{pmatrix}
 \cos t & -\sin t & 0 & 0 \\
 \sin t & \cos t & 0 & 0 \\
  0 & 0 & \cos t & -\sin t \\
  0 & 0 & \sin t & \cos t
 \end{pmatrix}
 \begin{pmatrix}
 x_1\cos t+v_1\sin t \\ x_2\cos t+v_2\sin t\\x_3\cos t+v_3\sin t \\ x_4\cos t+v_4\sin t
 \end{pmatrix}.
 \]
 
 \end{remark}
 
 Returning now to the general case, observe that $\mathcal P(0)=s(0)z=z=\rho(0)$ and
 $\dot{\mathcal P}(0)=iz+v=W_z+\dot{\rho}(0)$ as expected. Remark that in order to get $F$-admissible geodesics we need to impose condition $v\neq -iz$.
 
 We also observe that
 \[
 \langle \dot{\mathcal P}(0), W_z \rangle=1+\langle v,W_z\rangle >0,
 \]
 where $\langle\cdot , \cdot \rangle$ is the Euclidean inner product in $\R^{2n}$.

\subsubsection{The geodesic connectedness of the canonical strong Kropina manifold $(\Sph^{2n-1},F)$}

Let us fix two points $p$ and $q$ on $\Sph^{2n-1}$ and ask the question whether we can always join these two points by an $F$-geodesic.

Since $(\Sph^{2n-1},h)$ is the round sphere it has sectional curvature $K=1$ and the flow $\varphi$ is a smooth free  isometric action of
$\Sph^1$ on $\Sph^{2n-1}$, so it follows that the strong Kropina manifold   $(\Sph^{2n-1},F)$ with canonical metric is forward geodesically connected.


Firstly, in the case $p$ is not on the integral line of $W$ through $q$, we consider $q^-=\varphi(-l,q)$ on the flow through $q$, where $l=d(p,q)$, and let $\rho$ denote the $h$-geodesic from $p$ to $q^-$. Then $\mathcal P:[0,l]\to \Sph^{2n-1}$, $\mathcal P(t)=\varphi(t,\rho(t))$ is the desired $F$-geodesic from $p$ to $q$.

Secondly, if $p\in \varphi((-\infty,0),q)$, then the $F$-geodesic from $p$ to $q$ coincides as points set with $\varphi([-l,0],q)$.

Thirdly, for $p\in \varphi((0,\infty),q)$, clearly there is no geodesic going against the wind, but using the periodicity of $W$ we can see that
$p=\varphi(2\pi-l,q)$ and hence we can reduce this case to the precedent one.

\begin{remark}
Recall that in the case of a smooth effective almost-free (or pseudo-free) action of $\Sph^1$ on a manifold $M$ we can construct a Riemannian metric $\widetilde{h}$ such that the generated vector field $W$ is unit length.

In the case $M=\Sph^{2n-1}$, $n\geq 2$, for any $\varepsilon>0$ there exists navigation data $(h_\varepsilon,W)$, where
\begin{enumerate}[(i)]
\item $h_\varepsilon$ is a (real analytical) Riemannian metric on $\Sph^{2n-1}$ all of whose sectional curvatures differ from 1 by at most $\varepsilon$;
\item $W$ is a unit Killing vector field with closed integral lines
\end{enumerate}
(see Theorem 20 in \cite{BN1}). In this case, the constants $a_i$ in \eqref{canonical action} do not necessarily need to be all equal to one.

\end{remark}

We can conclude

\begin{theorem}
For any pseudo-free action of  $\Sph^1$ on $\Sph^{2n-1}$ there exists a strong Kropina metric on $\Sph^{2n-1}$ which is forward complete and geodesically connected, respectively.
\end{theorem}

\begin{remark}
The  forward complete, forward geodesically connected canonical strong Kropina manifold $\Sph^{2n-1}$ have the properties
\begin{enumerate}[(i)]
\item All geodesics of $F$ are closed.
Indeed, one can easily see that
\begin{equation*}
\begin{cases}
\mathcal P(0)=\mathcal P(\pi)=\mathcal P(2\pi)=z, \hspace{0.2in}\\
\dot{\mathcal P}(0)=\dot{\mathcal P}(\pi)=\dot{\mathcal P}(2\pi)=iz+v.
\end{cases}
\end{equation*}
( If $a_1,\dots,a_n$ are not equal to 1, then they must be rational numbers, such that all geodesics of $F$ close.)

\item It is of constant positive sectional curvature.
\item For the canonical strong Kropina metric, one can see that the $F$-length of a great circle is $\pi$.

\item In the case $\rho(t)$ is an $h$-geodesic of $\Sph^{2n-1}$ invariant under the flow of $W$, then the resulting Kropina geodesic is just a
reparametrization of $\rho$.
\end{enumerate}
\end{remark}

If we take into account that for a given point $z\in \Sph^{2n-1}$, its  $h$-conjugate point, that is the $h$-cut locus of $z$, is the antipodal point and Corollary \ref{cor: F-cut locus}, we obtain

\begin{proposition}
	In the canonical strong Kropina space $(\Sph^{2n-1},F)$, the $F$-conjugate and the $F$-cut point of a point $z\in \Sph^{2n-1}$ is the point $z$ itself. 
\end{proposition}

\begin{remark}
	Observe that this proposition provides a counter-example to the {first property} of the cut locus of a classical Finsler manifold enumerated in Remark \ref{rem: classical cut locus prop}.
	
\end{remark}


 \subsection{Locally Euclidean spaces}
 
 We will consider now the case of locally Euclidean spaces, that is Riemannian manifolds $(M,h)$ of zero sectional curvature.
 
 In this case it is known that $M$ is the Euclidean space $\R^n$ or is a quotient space $\R^n\setminus \Gamma$ of the Euclidean space by the action of some discrete free isometry group $\Gamma$ (see \cite{W}). In the two dimensional case the cylinder, the torus, the M\"obius band and Klein bottle enters in this category. Observe that our results in Section 3 show that in the case of zero sectional curvature, $M$ admits a strong Kropina metric if and only if $W$ is parallel on $(M,h)$.
 
 We will recall some basics about locally Euclidean spaces.
 
 {\bf 1. The case $(M,h)=(\R^n,\delta)$.}
 
 It is known that the full isometries group of $(\R^n,\delta)$ is isomorphic to the semi-direct product group
 \[
 O(n)\ltimes V^n =\{(A,a):A\in O(n),a\in V^n
 \}
 \]
 of combined rotations and translations. 
Since the vector group $V^n$ of parallel translations in $\R^n$ is a normal subgroup of the isometries group 
 $Isom(\R^n,\delta)$ of $(\R^n,\delta)$ it is natural to define the epimorphism 
 \[
 d:Isom(\R^n,\delta) \to Isom(\R^n,\delta)\slash V^n=O(n).
 \]

 An isometry $f\in Isom(\R^n,\delta)$ can be written as
 \[
 f=(A,a)\in O(n)\ltimes \R^n
 \]
 in the sense that $f:\R^n\to \R^n$, $x\mapsto f(x)=A\cdot x+a$.
 
 It is also known that an arbitrary Killing field on   $(\R^n,\delta)$ is given by $X=(C,w)$, where $w\in \R^n$ and $C$ is a skew-symmetric $n\times n$ matrix. Recall that Proposition 9 in \cite{BN1} tells us that the Killing field $X=(C,w)$ is invariant under the isometry $f=(A,a)$ if and only if
 \begin{equation}\label{invar_cond_Killing}
 \begin{cases}
 [A,C](x)=0,\textrm{ for all } x\in \R^n, \textrm{ and }\\
 C\cdot a+w=A\cdot w.
 \end{cases}
 \end{equation}
 
 In particular, if $C=0$ (i.e. $X$ is a translation only), then the condition \eqref{invar_cond_Killing} simplifies to
 $A\cdot w=w$.

 {\bf 2. The case $(M,h)=(\R^n\setminus \Gamma,h)$.}
 
 Here $\Gamma\subset Isom(\R^n,\delta)$ is a discrete free isometry subgroup. In this case, the existence of a unit length non-trivial Killing field on $M$ is equivalent to the existence of a non-zero vector $a\in \R^n$, which is invariant under the group $d\Gamma=\{df:f\in \Gamma\}$, where $d$ is the epimorphism defined above  (see \cite{BN1} for details).
 Each such vector field on $M$ is the projection of a parallel vector field $X=(0,a)$ on $\R^n$, where $a$ is invariant under $d\Gamma$ (see Proposition 10 in \cite{BN1} ).
 
 \begin{remark}
 There exist three dimensional compact orientable locally Euclidean manifolds $M$ all of whose Killing fields are trivial.
 \end{remark}
 
 We recall that among locally Euclidean spaces, the spaces of the form $\R^m\times T^l$, where $T^l$ is an $l$-dimensional torus are symmetric Riemannian manifolds, so they cannot admit quasi-regular Killing fields of constant length. 
 
 In the two dimensional case, it is known that the M\"obius band and Klein bottle admit such fields along many other non-homogeneous locally Euclidean spaces. All these admit geodesically connected strong Kropina metric structures. However, the flat cylinder and torus also admit such structures as will be seen.

 In general, a locally Euclidean space $M=\R^n\setminus \Gamma$ has a quasi-regular unit Killing field if and only if there exists an isometry $f=(A,(I-A)b+a)\in \Gamma$, that is $f:M\to M$, $x\mapsto f(x)=A\cdot x+(I-A)b+a$, where $I$ is the identical linear transformation, $a\neq 0$ is invariant under $d\Gamma$, $b$ and $a$ are orthogonal, and $A\neq I$ has finite order. The corresponding Killing field is $d\pi(0,a)$, where $\pi:\R^n\to M$ is the canonical projection  (see Theorem 25 in \cite{BN1}).

 We have 
 \begin{theorem}\label{thm: quasi-regular W for locally euclidean}
 	Let $(M,F)$ be a strong Kropina manifold with navigation data $(h,W)$, where $(M=\R^n\slash\Gamma,h)$ is a locally Euclidean space, $\Gamma$ a free discrete subgroup of $Isom(\R^n,\delta)$, $W$ a  unit Killing field on $(M,h)$, and $p\in M$  a fixed point.
 	\begin{enumerate}[(i)]
 		\item If $W$ is quasi-regular, then $p$ can be joined by an $F$-geodesic to any other point $q\in M$.
 		\item Otherwise, $p$ can be joined by an $F$-geodesic to other point $q\in M$ if and only if there exists $\widetilde{r}\in \pi^{-1}(q)$ such that
 		$\langle \widetilde p\widetilde r,\widetilde W_{\widetilde p}\rangle>0$, where  $\pi:{\R^n\to M=\R^n\slash \Gamma}$ is the canonical projection, $\widetilde{p}$ and $\widetilde{W}$ are corresponding in $\R^n$ to $p$ and $W$. 
 	\end{enumerate}
 \end{theorem}

We remark that in this case, since $W$ is parallel, $W^\perp$ is integrable distribution 
(see Proposition \ref{integrable_distrib}).

 We will consider in the following some special locally Euclidean spaces.
 \subsubsection{The Euclidean space}
 
 We have already introduced this case in Subsection \ref{subsec: Euclidean space}. We only summarize here the results.
 
 \begin{proposition}
 	Let $(\R^n,F)$ be the strong Kropina manifold with navigation data $(h,W)=(\delta, C)$, where $\delta$ is the canonical Euclidean metric and $C=(C^1,C^2,\dots,C^n)$. For any pair of points $p,q\in \R^n$, we can join $p$ to $q$ by an $F$-geodesic if and only if $q\in \mathcal D_p^+:=\{x\in \R^n:\langle x-p,C \rangle >0 \}$, and by an $F$-geodesic from $q$ to $p$ if and only if $q\in \mathcal D_p^-:=\{x\in \R^n:\langle x-p, C \rangle <0 \}$. If $q\in \{x\in \R^n:\langle x-p, C \rangle =0 \}$, then we cannot join $p$ to $q$ nor $q$ to $p$ by any $F$-geodesic.
 \end{proposition}
 
 \begin{remark}
 	It is trivial to see that in this case, the cut locus of a point $p\in M$ is empty. 
 	
 \end{remark}
 
 \subsubsection{The cylinder $M=\Sph^1\times \R^n$}
 
 {\bf 1. The Riemannian flat cylinder}
 
 Let us consider the flat (straight) cylinder $M=\Sph^1\times \R$ endowed with the canonical metric $h$ induced from $(\R^3,\delta)$.
 
 We parametrize the surface  $M=\Sph^1\times \R$ by
 \[
 M=\{(\cos u, \sin u, v)\in \R^3 : 0\leq u \leq 2\pi, v\in \R \},
 \]
 with local coordinates $(u,v)$. Using the canonical embedding
\[
x=\cos u,\ y=\sin u,\ z=v
\]
we obtain the induced Riemannian metric $ds^2=(dx)^2+(dy)^2+(dz)^2=(du)^2+(dv)^2$. Hence the induced metric is
the usual flat metric
$h=\begin{pmatrix}
1 &0 \\ 0 & 1
\end{pmatrix}
$.

We recall that the curves
\begin{enumerate}[(i)]
	\item $\{v=v_0=constant\}$ 
                    are called {\it parallels}. They are unit circles in $\R^3$;
	\item $\{u=u_0=constant\}$ 
                          are called {\it meridians}. They are straight lines in  $\R^3$.
\end{enumerate}

Since the metric $h$ is flat, the $h$-geodesic equations are
\begin{equation}\label{rho for flat cylinder}
\rho(s)=(a^1s+p^1,a^2s+p^2),
\end{equation}
where $p=(p^1,p^2)$ is the initial point of $\rho$ and $\dot \rho=a=(a^1,a^2)$ is the initial direction. The unit length parametrization gives the condition $(a^1)^2+(a^2)^2=1$. Regarded as curves in $\R^3$, these $h$-geodesics are the helices
\[
\rho(s)=(\cos(a^1s+p^1), \sin(a^1s+p^1),a^2s+p^2) =(e^{(a^1s+p^1)i},a^2s+p^2).
\]

It is in many cases useful to treat geodesics on $M$ by using its universal covering $\pi:\widetilde{M}\to M$, $(\widetilde{M}:=\R\times \R,\widetilde{h})$, where the $\widetilde{h}$-geodesics are straight lines.


\bigskip

{\bf 2. Kropina metrics on the straight cylinder}

With same notations as above, we will consider the strong Kropina metric $(M,F)$ on the straight cylinder with navigation data $(h,W)$, where $W$ is the {unit} Killing field $W(u,v)=A\frac{\partial}{\partial u}|_{(u,v)}+B \frac{\partial}{\partial v}|_{(u,v)}$, where $A$, $B$ are real constants, $A^2+B^2=1$.
The induced flow is
$
\varphi_t(u,v)=(u+At,v+Bt),
$
 $t\in \R$. Regarded as a curve in $\R^3$, the flow reads $\varphi_t(\cos u, \sin u, v)=\varphi_t(e^{iu},v)=(e^{i(u+At)},v+Bt)$.
 
 If $\rho:(-\varepsilon, \varepsilon)\to M$, given by \eqref{rho for flat cylinder}, is an $h$-unit geodesic, then the corresponding Kropina geodesic is given by
 \begin{equation}\label{F-geod on flat cylinder}
 \mathcal P(s)=\varphi _s(\rho(s))=((a^1+A)s+p^1, (a^2+B)s+p^2),\quad (a^1,a^2)\neq (-A,-B).
 \end{equation}
 Obviously, this is an $F$-unit speed geodesic.

 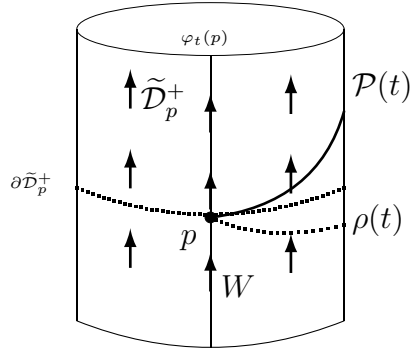
\begin{figure}[H]
 	\begin{picture}(200,140)
 	
 	\put(200, 130){\ellipse{100} {20}}
 	\put(150,20){\line(0,1){110}}
 	\put(250, 20){\line(0,1){110}}
 	\put(150,20){\qbezier(0, 0) (50, -20) (100,0)}
 	{\linethickness{1pt}
 		\put(150,40){\qbezier[50](0, 30) (50, 10) (100,30)}}
 	\put(200, 10){\line(0, 1){110}}
 	\put(125,70){{\tiny $\partial\widetilde{\mathcal D}_p^+$}}
 	\put(170, 100){ $\widetilde{\mathcal D}_p^+$}
 	\linethickness{1pt}
 	\put(170, 40){\vector(0,1){15}}
 	\put(200,31){\vector(0,1){15}}
 	\put(230, 38){\vector(0,1){15}}
 	\put(170, 70){\vector(0,1){15}}
 	\put(200,62){\vector(0,1){15}}
 	\put(230, 68){\vector(0,1){15}}
 	\put(170, 100){\vector(0,1){15}}
 	\put(200,91){\vector(0,1){15}}
 	\put(230, 98){\vector(0,1){15}}
 	\put(185,49){ $p$}
 	\put(200,29){  $W$}
 	\put(200,59){\circle*{4}}
 	\put(185,125){ {\tiny $\varphi_t(p)$}}
 	
 	\put(200,59){\qbezier[20](0, 0) (20, -10) (50,-3)}
 	\put(253, 55){$\rho(t)$}
 	\put(200,59){\qbezier(0, 0) (40, 5) (50,40)}
 	\put(253, 105){$\mathcal P(t)$}
 	\end{picture}
 	\caption{The unit Killing vector field $W=(0,1)$ generated by a parallel translation along the meridians.}\label{fig: W^H-geodesics on cylinder}
 \end{figure}%

 In the universal covering we have $W=(A,B)$ and $W^\perp=(-B,A)$ (as vector field), i.e. $W\perp W^\perp$, and hence, for any $h$-geodesic $\rho(s)$ with $\dot{\rho}(0)\neq -W$, the corresponding Kropina geodesic $\mathcal P(s)=\varphi_s(\rho(s))$ must stay in the half plane, determined by the straight line $W^\perp$, that contains $W$.

 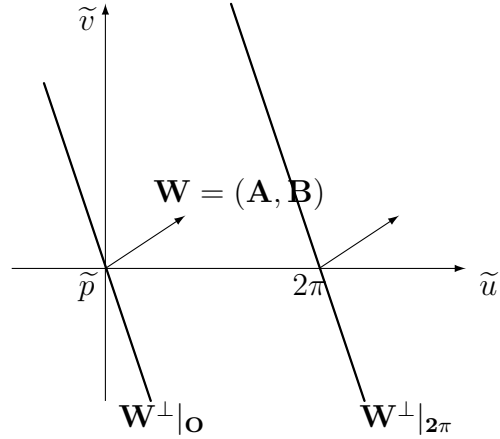
\begin{figure}[H]
 	\begin{picture}(200,180)
 	\put(155,10){\vector(0,1){150}}
 	\put(120,60){\vector(1,0){170}}
 	\put(250, 0){${\bf W^\perp |_{2\pi}}$}
 	\put(235,60){\vector(3,2){30}}
 	\put(155,60){\vector(3,2){30}}
 		\put(145,50){$\widetilde{p}$}
 			\put(225,50){$2\pi$}
 				\put(295,50){$\widetilde{u}$}
 				\put(145, 150){$\widetilde{v}$}
 	\put(172,10){\line(-1, 3){40}}
 	\put(173, 85){${\bf W=(A,B)}$}
 	\put(160, 0){${\bf W^\perp |_{O}}$}
 	\put(252,10){\line(-1, 3){50}}
 	\thicklines
 	\put(252,10){\line(-1, 3){50}}
 	\put(172,10){\line(-1, 3){40}}
 	\end{picture}
 	\caption{$W$ and $W^\perp$ on the universal covering $\widetilde M$ of the straight cylinder.}\label{figure 4}
 \end{figure}%
 
 Remark that
 \begin{enumerate}[(i)]
 	\item the case $(A,B)=(0,1)$
                            gives a strong Kropina manifold generated by a Killing field $W$ generated by a translation along meridians (see Figure \ref{fig: W^H-geodesics on cylinder});
 	\item the case $(A,B)=(1,0)$
                             gives a strong Kropina manifold generated by a Killing field $W$ generated by a  rotation along parallels (see Figure \ref{fig: W^S-geodesics on cylinder}).
 \end{enumerate}

 \begin{figure}[H]
 	\begin{picture}(200,140)
 	\put(200, 130){\ellipse{100} {10}}
 	\put(150,20){\line(0,1){110}}
 	\put(250, 20){\line(0,1){110}}
 	\put(150,20){\qbezier(0, 0) (50, -20) (100,0)}
 	\linethickness{1pt}
 	\put(150, 20){\qbezier[80](0, 0) (40, -10) (100,20)}
 	\put(150, 60){\qbezier[80](0, 0) (40, -10) (100,20)}
 	\put(150, 100){\qbezier[80](0, 0) (40, -10) (100,20)}
 	\put(150, 20){\vector(4,-1){50}}
 	\put(170, 17){\vector(1,0){50}}
 	\put(194,20){\vector(4,1){50}}
 	\put(240, 35){\vector(3,2){20}}
 	\put(150, 60){\vector(4,-1){50}}
 	\put(170, 57){\vector(1,0){50}}
 	\put(194,60){\vector(4,1){50}}
 	\put(240, 75){\vector(3,2){20}}
 	\put(150, 100){\vector(4,-1){50}}
 	\put(170, 97){\vector(1,0){50}}
 	\put(194,100){\vector(4,1){50}}
 	\put(240, 115){\vector(3,2){20}}
 	\put(230,60){$\varphi_t(p)$}
 	\put(170,45){$p$}
 	\put(175,57){\circle*{4}}
 	\put(200,42){$\widetilde{W}_p$}
 	\put(220,5){$\rho_+(t)$}
 	\put(120,80){$\rho_-(t)$}
 	\put(170, 45){\qbezier[80](-20,50) (0, 2) (60,-32)}
 	
 	\put(175, 57){\qbezier[30](0, 0) (28, 35) (35,70)}
 	\put(195, 130){$\rho(t)$}
 	\put(175, 57){\qbezier(0, 0) (57, 35) (65,70)}
 	\put(230, 130){$\mathcal P(t)$}
 	\end{picture}
 	\caption{The unit Killing vector field $W$ generated by a spiral rotation}
 	\label{fig: W^S-geodesics on cylinder}
 \end{figure}
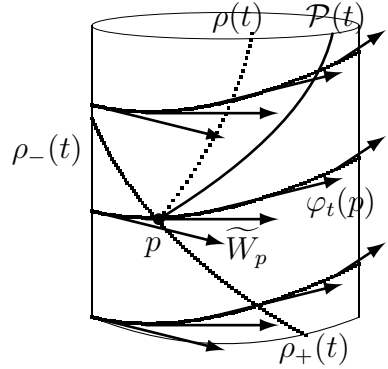%
 

 We obtain
 
 \begin{proposition}\label{Prop 8.12}
 	Let $\widetilde{p}=(0,0)\in \widetilde M$, and $\widetilde{q}\in \widetilde{M}$, $\widetilde{q}\neq \widetilde{p}$. Then
 	\begin{enumerate}[(i)]
 		\item If $B\neq 0$, then
 		\begin{enumerate}[(a)]
 			\item if $\widetilde{q}\in \mathcal{D}^+_{\widetilde{p}}=\{\widetilde{x}=(\widetilde{u},\widetilde{v})\in  \widetilde{M}:A\widetilde{u}+B\widetilde{v}>0 \}$, then there exists an $F$-geodesic from $\widetilde{p}$ to $\widetilde{q}$;            

 			\item if $\widetilde{q}=(\widetilde{u},\widetilde{v})\in W^\perp_{{\tilde{p}}}$, then there is no $F$-geodesic from $\widetilde{p}$ to      $\widetilde{q}$, nor from $\widetilde{q}$ to $\widetilde{p}$;

                    \item  if $\widetilde{q}\in \mathcal{D}^-_{\widetilde{p}}=\{\widetilde{x}=(\widetilde{u},\widetilde{v})\in 		      \widetilde{M}:A\widetilde{u}+B\widetilde{v}<0 \}$, then there is no $F$-geodesic     
                    from $\widetilde{p}$ to      $\widetilde{q}$, but there exists an $F$-geodesic from $\widetilde{q}$ to $\widetilde{p}$.

 		\end{enumerate}
 		\item If $B=0$, then there exists an $F$-geodesic from $\widetilde{p}$ to $\widetilde{q}$, for any $\widetilde{q}\in \mathcal{D}^+_{\widetilde{p}}=\{\widetilde{x}=(\widetilde{u},\widetilde{v})\in \widetilde{M}:\widetilde{u}>0 \}$.

 	\end{enumerate}
 	
 \end{proposition}
 
 Indeed, in the case $B\neq 0$, trajectories of $W$ are not closed curves and the conclusion follows from Theorem \ref{thm: quasi-regular W for locally euclidean}, (ii). 
 
 If $B=0$, we have $|A|=1$, that is the flow of $W$ is a rotation along the parallels.  In this case, all trajectories of $W$ are closed curves on the surface of the straight cylinder and hence the conclusion follows from Theorem \ref{thm: quasi-regular W for locally euclidean}, (i).
 
{
One should be careful at the following points when inducing the properties of  $F$-geodesics on the canonical strong Kropina space $(\Sph^1\times \R, F)$ from those of $F$-geodesics on the universal covering space $\widetilde{M}$.
Even though on the universal covering space $\widetilde{M}$, the statement (i), (b) in Proposition \ref{Prop 8.12} holds, on $(\Sph^1\times \R, F)$ the following two cases occur.

For a point $p\in \Sph^1\times \R$,
\begin{enumerate}[(i)]
\item if $|B|=1$, for any point $q\in \pi(W_{\widetilde{p}}^\perp)$ there is no $F$-geodesic from $p$ to $q$.
\item if $|B|\ne 1$, for any point $q\in \pi(W_{\widetilde{p}}^\perp)$ there are infinitely many  $F$-geodesics from $p$ to $q$.
\end{enumerate}

}

 \begin{corollary}
 		Let $(\Sph^1\times \R ,F)$ be the canonical Kropina manifold constructed above with navigation data $(h,W=(A,B))$. 
 	\begin{enumerate}[(i)]
 		\item If {$|B|\ne 1$},
then $(\Sph^1\times \R     ,F)$ is geodesically connected, i.e. we can join any two points on  $\Sph^1\times \R$ by an $F$-geodesic segment. 
 			\item If  {$|B|=1$},
then there are points on $\Sph^1\times \R$ that cannot be joined by an $F$-geodesic. 
 	\end{enumerate}	
 \end{corollary}
 
In the case there is an $F$-geodesic { $\tilde{\mathcal{P}}$} from $\widetilde{p}$ to $\widetilde{q}$, we can determine explicitly the initial conditions for ${\widetilde{\mathcal{P}}}$.

\begin{proposition}
	Let $(M,F)$ be a strong Kropina metric on the cylinder $\Sph^1\times \R$ with navigation data $(h,W)$, and let  ${\widetilde{p}}=(p^1,p^2)$, ${\widetilde{q}}=(q^1,q^2)\in M$
{be} two points that can be joined by an $F$-geodesic { $\widetilde{\mathcal{P}}$}. Then {$\widetilde{\mathcal{P}}$} has the initial conditions
	\begin{equation}
	\begin{cases}
	{ \widetilde{\mathcal{P}}}(0)={ \widetilde{p}},\\
	{  \dot{\widetilde{\mathcal{P}}}   (0)=\frac{2\langle W_{\widetilde{p}},\widetilde{p}\widetilde{q}\rangle}{||\widetilde{p}\widetilde{q}||_h^2}\cdot \widetilde{p}\widetilde{q},  }
	\end{cases}
	\end{equation}
	where ${\widetilde{p}\widetilde{q}}:=(q^1-p^1,q^2-p^2)$.
\end{proposition}
 \begin{proof}
Taking into account the equation of the $h$- and $F$-geodesics \eqref{rho for flat cylinder} and \eqref{F-geod on flat cylinder}, respectively, observe that {$\widetilde{q}$} is on {$\widetilde{\mathcal{P}}$}, that is there exists a parameter value $s_0$ such that {   $\widetilde{q}=\widetilde{\mathcal{P}}(s_0)$},
if and only if
\begin{equation*}
\begin{cases}
(a^1+A)s_0+p^1=q^1,\\
(a^2+B)s_0+p^2=q^2,
\end{cases}
\end{equation*}
or, equivalently,
 	\begin{equation}\label{a*s formulas}
 	\begin{cases}
 	a^1s_0=q^1-p^1-As_0,\\
 	a^2s_0=q^2-p^2-Bs_0.
 	\end{cases}
 	\end{equation}
 	Taking now into account that $(a^1)^2+(a_2)^2=1$, the formulas \eqref{a*s formulas} imply
 	\[
 	s_0^2=(a^1s_0)^2+(a_2s_0)^2=(q^1-p^1)^2+(q^2-p^2)^2
 	-2\bigl[A(q^1-p^1)+B(q^2-p^2)\bigl]s_0+(A^2+B^2)s_0^2.
 	\]
 Since $W$ is unit, i.e. $	A^2+B^2=1$, the second order term $s_0^2$ reduces and we obtain
 \begin{equation}\label{s_0 on straight cylinder}
 s_0=\frac{(q^1-p^1)^2+(q^2-p^2)^2}{2\bigl[A(q^1-p^1)+B(q^2-p^2)\bigl]},
 \end{equation}
 	or, equivalently, $s_0=\frac{||  {\widetilde{p}\widetilde{q}}||^2}{2\langle W,  {\widetilde{p}\widetilde{q}} \rangle}$, where we regard ${ \widetilde{p}\widetilde{q}}:=(q^1-p^1,q^2-p^2)$ as the vector field with origin {$\widetilde{p}$} and tip {$\widetilde{q}$} in the universal covering $\widetilde{M}=\R^2$.
 	
 	By substituting \eqref{s_0 on straight cylinder} in  \eqref{a*s formulas} we obtain
 	\begin{equation*}
 	\begin{cases}
 	a^1=\frac{q^1-p^1}{s_0}-A=(q^1-p^1)\frac{2\langle W,  { \widetilde{p}\widetilde{q}} \rangle}{||{  \widetilde{p}\widetilde{q}   }||^2}-A,\\
 	a^2=\frac{q^2-p^2}{s_0}-B=(q^2-p^2)\frac{2\langle W,  {   \widetilde{p}\widetilde{q}    } \rangle}{||   {  \widetilde{p}\widetilde{q}    }||^2}-B,
 	\end{cases}
 	\end{equation*}
 	and hence $a=(a^1,a^2)=\frac{2\langle W,  {  \widetilde{p}\widetilde{q}   } \rangle}{|| {  \widetilde{p}\widetilde{q}    }||^2}    {   \widetilde{p}\widetilde{q}   }-W_{\widetilde{p}}$. Remark that is the initial condition for the $h$-geodesic that is deformed by the flow of $W$ into {$\widetilde{\mathcal{P}}$}.
 	
 	Using now that ${    \dot{\widetilde{\mathcal{P}}}}(0)=\dot{\rho}(0)+W_{\widetilde{p}}=a+W_{\widetilde{p}}$, the formulas needed follow immediately.

 	$\qedd$
 \end{proof}
 
 \begin{remark}
 	It is important to remark that $s_0$ in formula \eqref{s_0 on straight cylinder} in the proof above is actually the $F$-distance between $p$ and $q$.
 	
 \end{remark}
 {\bf 3. The cut locus of a strong Kropina metric on the straight cylinder}

We will consider the cut locus of a strong Kropina space with navigation data $(h, W)$ on a straight cylinder.
The vector field $W$ is a unit Killing vector field represented by
\begin{eqnarray}\label{9.8}
W(u, v)=A\frac{\partial}{\partial u}+B\frac{\partial}{\partial v},
\end{eqnarray}
where $A^2+B^2=1$.

We consider the $F$-geodesics on the universal covering space $\widetilde{M}=\R \times \R$ of a straight cylinder.

Fix $\widetilde{p}=(0, 0)$ and consider the cut locus of $F$-geodesics emanating from $\widetilde{p}$.

 A generic unit speed $F$-geodesics $\mathcal{P}(s)$ on a straight cylinder can be written as
\begin{eqnarray*}
  \mathcal{P}(s)=(e^{(a^1+A)si}, (a^2+B)s),  \hspace{0.1in} (a^1, a^2)\ne (-A, -B).
\end{eqnarray*}
From (\ref{F-geod on flat cylinder}),  on  the universal covering space $\widetilde{M}$ this $F$-geodesic is represented by
\begin{eqnarray}\label{9.9}
  \widetilde{\mathcal{P}}(s)=\big((a^1+A)s, (a^2+B)s \big), \hspace{0.1in} (a^1, a^2)\ne (-A, -B).
\end{eqnarray}

\begin{remark}\label{rem: h-dist vs F-dist on cylinder}
	\begin{enumerate}
		\item The $h$-cut locus of $\widetilde{p}=(0,0)$ is the opposite meridian $\{\widetilde{u}=\pi\}$,
                                                      i.e. we can write 
		$\mathcal{C}_{\widetilde{p}}^h=\{(\pi,\widetilde{v}):\widetilde{v}\in \R\}$.     
                   		\item The $h$-distance from $\widetilde{p}$ to a point     $\widetilde{q}=(\pi, \widetilde{v}_0)\in \mathcal{C}_{\widetilde{p}}^h$ is $d_h(\widetilde{p}, \widetilde{q})=\sqrt{\pi^2+\widetilde{v}^2_0}$ 
		                        (the Theorem of Pythagoras). 
	\end{enumerate}
	
\end{remark}

\begin{theorem}Let $(M, F)$ be a strong Kropina space on a straight cylinder $\Sph^1 \times \R$ with navigation data $(h, W)$, where
	$W=A\frac{\partial}{\partial u}+B\frac{\partial}{\partial v}$, and let us denote by $\widetilde{M}$ the universal covering of $M$.
	
	The $F$-cut locus of the point  $\widetilde{p}=(0,0)\in \widetilde{M}$ is 
\begin{equation}\label{F-cut locus on cylinder}
\mathcal{C}_{\widetilde{p}}^F=\{(\pi+A\sqrt{\pi^2+\widetilde{v}^2},\widetilde{v}+B\sqrt{\pi^2+\widetilde{v}^2}):\widetilde{v}\in \R\}
\end{equation}

\end{theorem}
\begin{proof}
	Since the cut locus of $F$ is obtained by "twisting" the $h$-cut locus by the flow of $W$ (see Corollary \ref{cor: F-cut locus}), we have $\widetilde{q}\in \mathcal{C}_{\widetilde{p}}^F$ if and only if  $\widetilde{q}=\varphi_l(\hat{q})$, where $\hat{q}\in \mathcal{C}_{\widetilde{p}}^h$ and $l=d_h(\widetilde{p},\hat{q})$.
	
	It follows $\widetilde{q}\in \mathcal{C}_{\widetilde{p}}^F$ if and only if   $\widetilde{q}=\varphi_l(\pi, \widetilde{v}_0)=( \pi+A\sqrt{\pi^2+\widetilde{v_0}^2}, \widetilde{v_0}+B\sqrt{\pi^2+\widetilde{v_0}^2})$,
	where we have used $l=\sqrt{\pi^2+\widetilde{v_0}^2}$ 
	(see Remark \ref{rem: h-dist vs F-dist on cylinder}), and the flow action formula. As set points we obtain the formula in the theorem.

	$\qedd$
\end{proof}

Let us consider some special cases in the following.

\begin{corollary}
	Let $(M, F)$ be a strong Kropina space on a straight cylinder $\Sph^1 \times \R$ with navigation data $(h, W)$, where $W=\frac{\partial}{\partial v}$,
	and let     $\widetilde{M}$ be the universal covering space   of $M$. 
	The $F$-cut locus of the point  $\widetilde{p}=(0,0)\in \widetilde{M}$ is 
	\begin{equation}
	\mathcal{C}_{\widetilde{p}}^F=\{(\pi, \widetilde{v}+\sqrt{\pi^2+\widetilde{v}^2} ):\widetilde{v}\in \R\}=\{(\pi, \widetilde{v}); \widetilde{v}>0\}.	
	\end{equation}
	
	In other words, the $F$-cut locus of any point $p\in \Sph^1 \times \R$ is the half meridian through the antipodal point of $p$.

\end{corollary}
Indeed, in this case we have $(A,B)=(0,1)$, that is the flow of $W=\frac{\partial}{\partial v}$ through $p$ is simply a translation along the meridian passing through same point $p$.

Observe that the function $\psi:\R\to \R$, $\psi(x)=x+\sqrt{x^2+\pi^2}$ is continuous function and $\psi(\R)=(0,\infty)$. Therefore, if we denote 
$\psi :=\widetilde{v}+\sqrt{\pi^2+\widetilde{v}^2}$,
then 
\[
\mathcal{C}_{\widetilde{p}}^F=\{(\pi, \psi):\psi\in (0,\infty)\},
\]
that is the $F$-cut locus of $p$ is the opposite half meridian.

\begin{remark}
	\begin{enumerate}
		\item This result is in concordance with Corollary \ref{cor: F-cut locus}.
		\item 	Observe that in this case, the $F$-cut locus of a point $p$ satisfies the properties in Remark \ref{rem: classical cut locus prop} from the classical case. 
	\end{enumerate}
	
\end{remark}
\begin{corollary}	Let $(M, F)$ be a strong Kropina space on a straight cylinder $\Sph^1 \times \R$ with navigation data $(h, W)$, where $W=\frac{\partial}{\partial u}$,
	and let     $\widetilde{M}$ be the universal covering space   of $M$. 
	The $F$-cut locus of the point  $\widetilde{p}=(0,0)\in \widetilde{M}$ is 
\begin{equation}\label{F-cut locus on cylinder (A,B)=(1,0)}
\mathcal{C}_{\widetilde{p}}^F=\{(\pi+\sqrt{\pi^2+\widetilde{v}^2},\widetilde{v}):\widetilde{v}\in \R\}.
\end{equation}
\end{corollary}

Indeed, in this case we have $(A,B)=(1,0)$, 
that is the flow of $W=\frac{\partial}{\partial u}$ through {$\widetilde{p}$} is a rotation along the parallel passing through {$\widetilde{p}$}.

The $F$-geodesics   which are emanating from $\widetilde{p}=(0, 0)$ on  $\widetilde{M}$ are represented by
\begin{eqnarray*}
	\widetilde{\mathcal{P}}(s)=\big((a^1+1)s, a^2s \big), \hspace{0.1in} (a^1, a^2)\ne (-1, 0),
\end{eqnarray*}
and \eqref{F-cut locus on cylinder} gives \eqref{F-cut locus on cylinder (A,B)=(1,0)}, that is a hyperbola (see Figure \ref{fig:The $F$-cut locus in the case $(A,B)=(1,0)$ }).

\begin{figure}[H]
	\begin{picture}(200,270)
	\put(30, 100){\vector(1,0){380}}
	\put(200,10){\vector(0,1){250}}
	\put(205, 250){ {   $\widetilde{v}$}}
	\put(405, 92){  {  $\widetilde{u}$}}
	\put(50,10){\line(0,1){250}}
	\put(200,10){\line(0,1){250}}
	\put(350,10){\line(0,1){250}}
	\put(36,93){{\tiny $-2\pi$}}
	\put(266,93){{\tiny$\pi$}}
	\put(266,100){\circle*{2}}
	\put(341,93){{\tiny$2\pi$}}
	\put(200,100){\vector(1,0){19}}
	\put(205,105){   {\tiny $W$}  }
	
	\put(190,100){{\tiny {  $\widetilde{p}$}}}
	\put(200,100){\circle*{3}}

	\put(125,100){\circle*{2}}
	\put(115,93){  {\tiny $-\pi$}}
	
	
	
	\put(354,100){ \qbezier(100, 100) (-108, 0) (100,-100)}

	
	\put(219,100){\circle{39}}
	\put(219,100){\circle*{2}}

	\end{picture}
	\caption{The $F$-cut locus in the case $(A,B)=(1,0)$.}
	\label{fig:The $F$-cut locus in the case $(A,B)=(1,0)$ }
\end{figure}
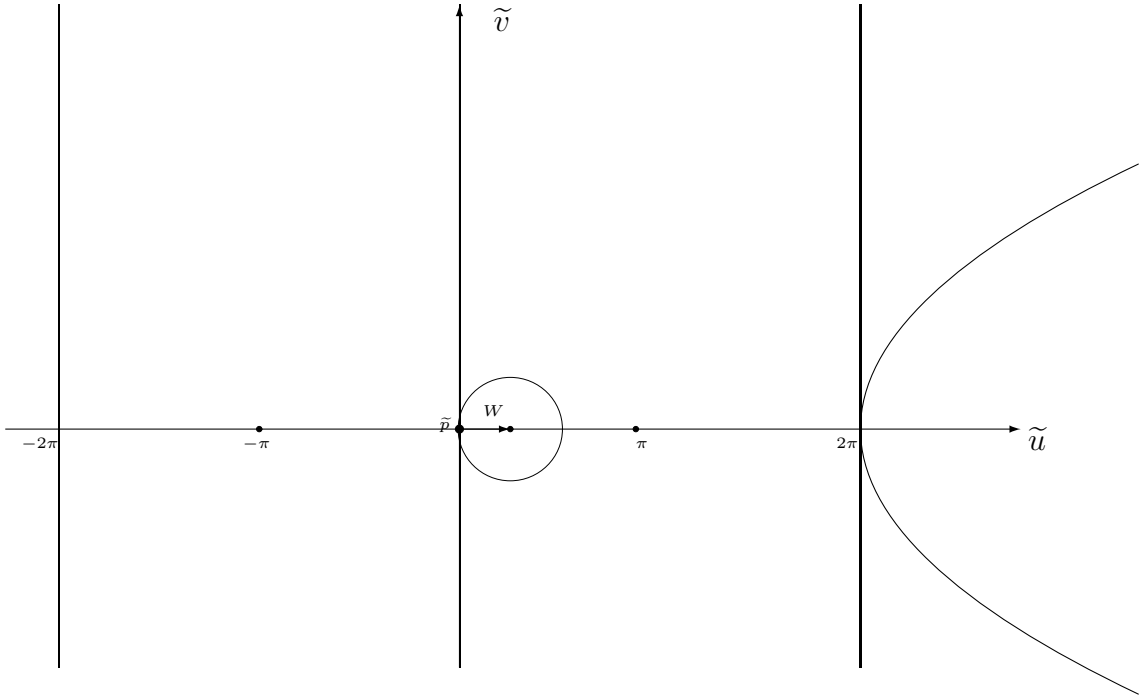%


\begin{remark}
	Let us recall that meridians {passing through $\widetilde{p}$} are rays for the canonical Riemannian metric $h$, that is they contain no $h$-cut points. It follows that the corresponding $F$-geodesics, obtained by twisting the meridians by means of the flow of $W=\frac{\partial}{\partial u}$ are $F$-rays, that is they cannot contain any $F$-cut points. In particular, these are the $F$-geodesics \begin{eqnarray*}
		\widetilde{\mathcal{P}}(s)=\big(s, \pm s \big), s>0,
	\end{eqnarray*}
that is { they are parallel to } the asymptotes 
{of}
 the hyperbola $\widetilde{u}=\pi+\sqrt{\pi^2+\widetilde{v}^2}$.
	
\end{remark}

  \subsubsection{The flat torus}

   { \bf (a) An $\Sph^1$-action on the flat torus}

In this section, we consider a flat torus
\begin{eqnarray*}
\Sph^1 \times \Sph^1=\{(z^1, z^2) | z^1=e^{\theta i}, z^2=e^{\eta i}, \theta, \eta \in \R\},
\end{eqnarray*}
with the flat canonical metric $h=(h_{ij})$  given by
\begin{eqnarray*}
   h\bigg( (z^1,z^2), (w^1, w^2)\bigg)=z^1\overline{w}^1+z^2\overline{w}^2,
\end{eqnarray*}
that is,
\begin{eqnarray*}
   ds^2=d\theta^2+d\eta^2,
\end{eqnarray*}
(for details see for instance \cite{Ca}).


We define an $\Sph^1$-action $\varphi_t$ by
\begin{eqnarray}\label{9.31}
     \varphi_t  :  \Sph^1 \times (\Sph^1 \times \Sph^1)  &\longrightarrow& \Sph^1 \times \Sph^1  \\
                           e^{ti} \times (z^1, z^2)   &\longmapsto&  (e^{\frac{1}{\sqrt{2}}ti}z^1,  e^{\frac{1}{\sqrt{2}}ti}z^2). \nonumber
\end{eqnarray}
Putting
\begin{eqnarray*}
   W_{(z^1, z^2)}:=\bigg(  \frac{d \varphi_t (z^1, z^2)}{dt}  \bigg)\bigg| _{t=0},
\end{eqnarray*}
we have
\begin{eqnarray}\label{9.32}
   W_{(z^1, z^2)}:=(\frac{1}{\sqrt{2}}iz^1, \frac{1}{\sqrt{2}}iz^2).
\end{eqnarray}

Since $\varphi_t$ is an isometry and $ |W_{(z^1, z^2)}|_h=1$ it means that $W$ is a unit Killing vector field, and hence we can consider
the Kropina manifold $(\Sph^1 \times \Sph^1, F)$
 constructed by the navigation data $(h, W)$.

{\bf (b) The geodesics on the flat torus $(\Sph^1 \times \Sph^1, h)$}

Since the torus is flat, the differential equations of the $h$-geodesics are
\begin{equation*}
\begin{cases}
\frac{d^2\theta}{dt^2}=0,\\
\frac{d^2\eta}{dt^2}=0.
\end{cases}
\end{equation*}
 
On the flat torus, the $h$-geodesic $\rho(t)=(e^{\theta(t)i}, e^{\eta(t)i})$ emanating from a point $\rho(0)=(e^{\theta(0)i}, e^{\eta(0)i})=(e^{b^1i}, e^{b^2i})$ with the tangent vector $\dot{\rho}(0)= (\frac{d\theta}{dt}(0)e^{b^1i},  \frac{d\eta}{dt}(0)e^{b^2i}) =(a^1e^{b^1i}, a^2e^{b^1i})$ is given by
\begin{equation*}
\begin{cases}
 \theta(t)=a^1t+b^1,\\
 \eta(t)=a^2t+b^2,
\end{cases}
\end{equation*}
that is,
\begin{eqnarray}\label{9.33}
  \rho(t)=\big( e^{( a^1t+b^1)i}, e^{(a^2t+b^2)i} \big),
\end{eqnarray}
where $(a^1)^2+(a^2)^2=1$.

{\bf (c) The cut locus of a flat torus.}
We will consider the cut locus of a flat torus $M=(\Sph^1 \times \Sph^1, h)$ on the universal covering space $\widetilde{M}=\R\times \R$.
If we denote the coordinates of   $\widetilde{M}=\R\times \R$ by $(u, v)$, then 
the universal covering map $\pi : \widetilde{M} \longrightarrow M=(\Sph^1 \times \Sph^1, h)$ is defined by
\begin{eqnarray*}
\pi : (u, v) \longrightarrow (e^{ui}, e^{vi}).
\end{eqnarray*}
 
Let us consider the $h$-geodesic
\begin{eqnarray}\label{9.34}
  \widetilde{\rho}(t)=\big( a^1t, a^2t \big), \hspace{0.1in} (a^1)^2+(a^2)^2=1,
\end{eqnarray}
emanating from $\widetilde{p}=(0, 0)\in \widetilde{M}$, that is the straight line passing through the origin $\widetilde{p}$.

Let us consider the square $\mathcal{S}:=\{(u, v)\in \R\times \R : 0\leqq u \leqq 2\pi,  \hspace{0.1in} 0\leqq u \leqq 2\pi\}$.
Since the four points $\widetilde{p}    {=}  (0, 0)$, $(0,2\pi)$, $(2\pi, 0)$ and $(2\pi, 2\pi)$ are identified each other, the cut locus is the set
\begin{eqnarray}\label{h:cut locus flat torus}
   \mathcal C_{\widetilde{p}}:=\{(\pi, v) :   0\leqq u \leqq 2\pi\} \cup \{(u, \pi) :  0\leqq u \leqq 2\pi \}.
\end{eqnarray}

Hence, we have
\begin{proposition}
Let $(\Sph^1 \times \Sph^1, h)$ be a flat torus.
Fix $p=\varphi(\widetilde{p})\in \Sph^1 \times \Sph^1$.
The $h$-cut locus of $p$ is the set ${\mathcal C}_p=\pi(\mathcal C_{\widetilde{p}})$.
\end{proposition}

\begin{remark}
	It is clear from topological reasons that the cut locus of a point $p$ on the flat torus $\Sph^1 \times \Sph^1$ must contain cycles. Intuitively, by identifying the opposite edges of the square, the cut locus closes and the cycles naturally come out.	
\end{remark}

\begin{figure}[H]
	\begin{picture}(200,150)
	\put(30, 30){\vector(1,0){150}}
	\put(50,10){\vector(0,1){150}}
	\put(40, 150){$\widetilde{v}$}
	\put(180, 22){$\widetilde{u}$}
	\put(100,10){\line(0,1){150}}
	\put(150,10){\line(0,1){150}}
       \put(30,80){\line(1,0){150}}
        \put(30,130){\line(1,0){150}}
	\put(40,23){{\tiny $\widetilde{p}$}}
       \put(50, 30){\circle*{2}}  
	\put(95,23){{\tiny$\pi$}}
	\put(100,30){\circle*{2}}
	\put(140,23){{\tiny$2\pi$}}
      \put(150,30){\circle*{2}}
      	\put(40,83){{\tiny$\pi$}}
	\put(50,80){\circle*{2}}
	\put(40,133){{\tiny$2\pi$}}
      \put(50,130){\circle*{2}}
       \put(105,85){$\mathcal{C}_{\widetilde{p}}$}
      \linethickness{1.5pt}
       \put(100,30){\line(0,1){100}}
         \put(50,80){\line(1,0){100}}
    \thinlines

   \put(320, 80){\ellipse{200}{100}}

  \put(120,20){\qbezier(150, 67) (153, 55),(200,55)}
  \put(120,20){\qbezier(200, 55) (247, 55),(250,67)}
 \put(120,20){\qbezier(155, 63) (158, 75),(200,75)}
 \put(120,20){\qbezier(200, 75) (242, 75),(245,63)}
   \put(120, 20){\qbezier(235, 60) (270,50), (250, 17)}
   
  \put(377, 60){$\circle*{2}$}
   \put(365,60){{\tiny$p$} } 
  \linethickness{1.5pt}

\put(120,20){\qbezier(150, 67) (155, 80),(200,80)}
\put(120,20){\qbezier(200, 80) (245, 80),(250,67)}

  \put(120,20){\qbezier[15](150, 67) (153, 50),(200,48)}
 \put(120,20){\qbezier[15](200, 48) (247, 50),(250,67)}

 \put(120,20){\qbezier(145, 101) (175, 110),(177,80)}
 \put(120,20){\qbezier(170, 73) (175, 73),(177,80)}

 \put(120,20){\qbezier[10](145, 101) (120, 90),(140,70)}
 \put(120,20){\qbezier[7](140, 70) (150, 60),(170,73)}
  \thinlines

  \put(300, 110){$\mathcal{C}_{p}$}

\put(190, 78){$\overset{\pi}{\longrightarrow}$}

	\end{picture}	
        \caption{}
\end{figure}%

{\bf (d) The geodesics of a strong Kropina  space on a flat torus} 

We will consider $F$-geodesics of a strong Kropina space  $(\Sph^1 \times \Sph^1, F)$ on a flat torus with navigation data 
$(h, W)$.
From (\ref{9.31}) and (\ref{9.33}), it follows that $F$-geodesics emanating from $(e^{b^1i}, e^{b^2i})$ are given by the form
\begin{eqnarray}\label{9.36}
\varphi_t(\rho(t))&=&(e^{\frac{1}{\sqrt{2}}ti} e^{ (a^1t+b^1)i},  e^{\frac{1}{\sqrt{2}}ti}e^{(a^2t+b^2)i})\\
                           &=&(e^{\{(\frac{1}{\sqrt{2}}+ a^1)t+b^1\}i},  e^{\{(\frac{1}{\sqrt{2}}+a^2)t+b^2\}i}) \nonumber
\end{eqnarray}
where $(a^1, a^2)\ne (-\frac{1}{\sqrt{2}},    -\frac{1}{\sqrt{2}})$.
This condition is obtained from 
\begin{eqnarray*}
\dot{\rho}(0)\ne -W_{(e^{b^1}, e^{b^2})},
\end{eqnarray*}
that is
\begin{eqnarray*}
(a^1ie^{b^1i}, a^2ie^{b^2i})\ne -(\frac{1}{\sqrt{2}}ie^{b^1i}, \frac{1}{\sqrt{2}}ie^{b^2i}).
\end{eqnarray*}

{\bf (e) The geodesic connected domain of a strong Kropina  space $(\Sph^1 \times \Sph^1, F)$}

On the universal covering space  $\widetilde{M}=\R\times \R$, let us define a square
\begin{eqnarray}
S_{(n, m)}:=\{(u, v)\in \widetilde{M} : 2n\pi\leqq u < 2(n+1)\pi, \hspace{0.1in}  2m\pi \leqq v <2(m+1)\pi \}
\end{eqnarray}
and denote a point $(2n\pi, 2m\pi)$ by $P_{(n,m)}$.

On each square $S_{(n, m)}$, we identify the oriented segment $\overrightarrow{P_{(n,m)}P_{(n+1, m)}}$ with  $\overrightarrow{P_{(n,m+1)}P_{(n+1, m+1)}}$ and 
$\overrightarrow{P_{(n,m)}P_{(n, m+1)}}$ with  $\overrightarrow{P_{(n+1,m)}P_{(n+1, m+1)}}$, that is we obtain a flat torus $\Sph^1 \times \Sph^1$.

On  $\widetilde{M} $, the $\Sph^1$-action $\varphi$ operates as
\begin{eqnarray}\label{9.38}
      \widetilde{\varphi} : e^t\times (u, v) \longmapsto (\frac{1}{\sqrt{2}}t+u, \frac{1}{\sqrt{2}}t+v)
\end{eqnarray}
for any $(u, v)\in \widetilde{M}$, hence the Killing vector field $\widetilde{W}$ is given by 
\begin{eqnarray}\label{9.39}
\widetilde{W}=(\frac{1}{\sqrt{2}}, \frac{1}{\sqrt{2}}).
\end{eqnarray}

It is not difficult to see that the projection of $\widetilde{W}$ into $\Sph^1\times\Sph^1$ is a quasi-regular unit Killing field. 

It follows (see Theorem \ref{thm: quasi-regular complete})
\begin{proposition}
	Let $(\widetilde{M}, F)$ be a strong Kropina space with navigation data $(h, \widetilde{W})$.
	For any point $\widetilde{q}   {=}  (u,v)\in \widetilde{M}$, where   $u+v>0$, there exists an $F$-geodesic emanating from $\widetilde{p}   {=}  (0, 0)$ to $\widetilde{q}$.
\end{proposition}	

	In other words, a strong Kropina space $(\Sph^1 \times \Sph^1, F)$ on a flat torus is geodesically connected. 

Without loss of generality, we may consider the geodesics emanating from $\widetilde{p}   {=}  (0, 0)$.
From  $(b^1, b^2)=(0, 0)$ and  (\ref{9.36}), they are can be written as
\begin{eqnarray}\label{9.40}
\widetilde{\mathcal P}(t)=((\frac{1}{\sqrt{2}}+a^1)t, (\frac{1}{\sqrt{2}}+a^2)t),
\end{eqnarray}
where $(a^1)^2+(a^2)^2=1$ and $(a^1, a^2)\ne (-\frac{1}{\sqrt{2}},    -\frac{1}{\sqrt{2}})$.

For any point $\widetilde{q}   {=}  (u, v)\in \widetilde{M}$,   $u+v>0$, we will consider the $F$-geodesics emanating from $\widetilde{p}   {=}  (0, 0)$ and passing through $\widetilde{q}$.
Then from (\ref{9.40}) we get
\begin{eqnarray*}
(\frac{1}{\sqrt{2}}+a^1)t&=&u, \\
(\frac{1}{\sqrt{2}}+a^2)t&=&v,
\end{eqnarray*}
where $(a^1)^2+(a^2)^2=1$ and $(a^1, a^2)\ne (-\frac{1}{\sqrt{2}},    -\frac{1}{\sqrt{2}})$.
From the above two equations, we have
\begin{eqnarray}
a^1t&=&u-\frac{1}{\sqrt{2}}t,\label{9.41}\\
a^2t&=&v-\frac{1}{\sqrt{2}} t. \label{9.42}
\end{eqnarray}
Substituting the above two equations to $(a^1t)^2+(a^2t)^2=t^2$, we have
\begin{eqnarray*}
  (u-\frac{1}{\sqrt{2}}t)^2+(v-\frac{1}{\sqrt{2}} t)^2=t^2.
\end{eqnarray*}
The above equation can be changed as follows:
\begin{eqnarray*}
u^2+v^2-\frac{2}{\sqrt{2}}(u+v)t=0.
 \end{eqnarray*}
Since $u+v>0$,  we get
\begin{eqnarray}\label{distance on the flat torus}
t=\frac{u^2+v^2}{\sqrt{2}(u+v)}.
\end{eqnarray}
Moreover, observe that from \eqref{9.41} , \eqref{9.42} and \eqref{distance on the flat torus} we get
\begin{eqnarray}
 a^1&=&\frac{u}{t}-\frac{1}{\sqrt{2}}=\frac{\sqrt{2}u(u+v)} {u^2+v^2}-\frac{1}{\sqrt{2}},\\
a^2&=&\frac{v}{t}-\frac{1}{\sqrt{2}}=\frac{\sqrt{2}v(u+v)} {u^2+v^2}-\frac{1}{\sqrt{2}},
\end{eqnarray}
 this is the initial direction for the $h$-geodesic that gives the $F$-geodesic joining the two points.

\begin{proposition}
	On the covering space $\widetilde{M}=\R\times \R$ of the flat torus $(\Sph^1\times \Sph^1,F)$, the $F$-length of a Kropina geodesic emanating from $\widetilde{p}   {=}  (0, 0)$ to a point 
	$\widetilde{q}   {=}  (u,v$)
	is given by
\begin{eqnarray*}
\frac{u^2+v^2}{\sqrt{2}(u+v)}.
\end{eqnarray*}
\end{proposition}

We can now easily determine the cut locus of the strong Kropina metric $F$ on the flat torus. Indeed, taking into account that the $\widetilde{h}$-cut locus is given by \eqref{h:cut locus flat torus}, then we can apply Corollary \ref{cor: F-cut locus} for each segment.

Indeed, let $\hat{q}=(u_0,\pi)\in \{(u, \pi) :  0\leq u \leq 2\pi \}$ be a point in the $\widetilde{h}$-cut locus of $\widetilde{p}=(0,0)$. Then equation \eqref{distance on the flat torus} implies that the $F$-distance from $\widetilde{p}$ to $\hat{q}$ is {$l_0=\frac{{u_0}^2+\pi^2}{\sqrt{2}({u_0}+\pi)}$}, and hence the set of points  $\{(u, \pi) :  0\leq u \leq 2\pi \}$ is transformed under the flow action into 
\begin{equation}
\{\varphi_{(u,\pi)}(l_0) :  0\leq u \leq 2\pi
\}=\{(\frac{u^2+\pi^2}{2(u+\pi)}+u,\frac{u^2+\pi^2}{2(u+\pi)}+\pi) :  0\leq u \leq 2\pi
\}.
\end{equation}

Similarly, the $F$-distance from $\widetilde{p}$ to a point in the set $\{(\pi,v) |  0\leq v \leq 2\pi \}$
 is $l_0=\frac{\pi^2+v^2}{\sqrt{2}(\pi+v)}$, and hence the set of points $\{(\pi,v) |  0\leq v \leq 2\pi \}$
 is transformed under the flow action into 
\begin{equation}
\{\varphi_{(\pi,v)}(l_0) :  0\leq v \leq 2\pi
\}=\{(\frac{\pi^2+v^2}{2(\pi+v)}+\pi,\frac{\pi^2+v^2}{2(\pi+v)}+v) :  0\leq v \leq 2\pi
\}.
\end{equation}

\begin{proposition}\label{Prop 9.19}
On the covering space $\widetilde{M}=\R\times \R$ of a flat torus $(\Sph^1\times \Sph^1, h)$, the cut locus of the origin is the set of points 
\begin{equation}
\mathcal{C}^F_{\widetilde{p}}=\{(\frac{u^2+\pi^2}{2(u+\pi)}+u,\frac{u^2+\pi^2}{2(u+\pi)}+\pi) :  0\leq u \leq 2\pi
\}\cup
\{(\frac{\pi^2+v^2}{2(\pi+v)}+\pi,\frac{\pi^2+v^2}{2(\pi+v)}+v) :  0\leq v \leq 2\pi
\}.
\end{equation}

\end{proposition}

Rewriting Proposition \ref{Prop 9.19}, we have
\begin{theorem}
Let $M=(\Sph^1 \times \Sph^1, h)$ be a flat torus.
And let $\widetilde{M}=\R\times \R$ and  $\pi :\widetilde{M} \longrightarrow M$ be a covering space of $M$ and the covering map, respectively.
Denote the vector field generated by  an $\Sph^1$-action $\varphi_t$
\begin{eqnarray*}
     \varphi_t  :  \Sph^1 \times (\Sph^1 \times \Sph^1)  &\longrightarrow& \Sph^1 \times \Sph^1  \\
                           e^{ti} \times (z^1, z^2)   &\longmapsto&  (e^{\frac{1}{\sqrt{2}}ti}z^1,  e^{\frac{1}{\sqrt{2}}ti}z^2) \nonumber
\end{eqnarray*}
by $W$, and the Kropina metric  constructed by navigation data $(h, W)$ by $F$.

Then the $F$-geodesic $\mathcal P(t)$ emanating from $p=\pi(\widetilde{p})$ of the Kropina space $(M, F)$, {where $\widetilde{p}=(0, 0)$,} is expressed by
\begin{eqnarray*}
\mathcal P(t) &=&(e^{(\frac{1}{\sqrt{2}}+ a^1)i},  e^{(\frac{1}{\sqrt{2}}+a^2)ti}) \nonumber
\end{eqnarray*}
where $(a^1, a^2)\ne (-\frac{1}{\sqrt{2}},    -\frac{1}{\sqrt{2}})$.

The $F$-cut locus of $p$ is the set $\pi(\mathcal{C}^F_{\widetilde{p}})$.
\end{theorem}


\section{Appendix. }
\subsection{ Conditions for a  Kropina metric  $F=\frac{\alpha^2}{\beta}$ to be projectively equivalent to the Riemannian metric $\alpha$ }
{
Let $(M, F=\alpha^2/\beta : A \longrightarrow \R^+)$ be a Kropina space, where $\alpha=\sqrt{a_{ij}(x)y^iy^j}$ and $\beta=b_i(x)y^i$.
We denote by the symbol $(;)$ the covariant derivative with respect to Levi-Civita connection on the Riemannian space $(M, \alpha)$.
The following notations are customary:
\begin{eqnarray*}
r_{ij}:=\frac{b_{i;j}+b_{j;i}}{2}, \hspace{0.1in} s_{ij}:=\frac{b_{i;j}-b_{j;i}}{2},\\
{s^i}_j:=a^{ik}s_{kj}, \hspace{0.1in} s_j=b^is_{ij}, \\
r_{00}:=r_{ij}y^iy^j, \hspace{0.1in}    {s^i}_0=   {s^i}_jy^j,   \hspace{0.1in}            s_0:=s_iy^i.
\end{eqnarray*}
}

The geodesic spray coefficients $\overline{G}^i$ of a Kropina space {  $(M, F=\alpha^2/\beta : A \longrightarrow \R^+)$}
 are expressed by 
\begin{eqnarray*}
2\overline{G}^i={{\gamma_0}^i}_0+2B^i,
\end{eqnarray*}
where 
\begin{eqnarray*}
B^i=-\frac{\beta r_{00}+\alpha^2 s_0}{b^2\alpha^2}y^i-\frac{\alpha^2 {s^i}_0}{2\beta}+\frac{\beta r_{00}+\alpha^2 s_0}{2b^2\beta}b^i
\end{eqnarray*}
and ${{\gamma_j}^i}_k$ are the coefficients of Levi-Civita connection with respect to $\alpha$.

We will get the conditions for the Kropina metric $F$ to be projectively equivalent to the Riemannian metric $\alpha$.
 Suppose that the Kropina metric is projectively equivalent to Riemannian metric $\alpha$.
Then there exists a function $P$ on $TM$, which is positively homogeneous of degree one with respect to $y$, such that
\begin{eqnarray}\label{10.4}
-\frac{\beta r_{00}+\alpha^2 s_0}{b^2\alpha^2}y^i-\frac{\alpha^2 {s^i}_0}{2\beta}+\frac{\beta r_{00}+\alpha^2 s_0}{2b^2\beta}b^i=Py^i,
\end{eqnarray}
(see \cite{BCS} and \cite{M1}).

Transvecting (\ref{10.4}) by $y_i$, we get
\begin{eqnarray*}
-\frac{\beta r_{00}+\alpha^2 s_0}{b^2}+\frac{\beta r_{00}+\alpha^2 s_0}{2b^2}=P\alpha^2,
\end{eqnarray*}
that is,
\begin{eqnarray}\label{10.5}
P=-\frac{\beta r_{00}+\alpha^2 s_0}{2b^2\alpha^2}.
\end{eqnarray}

Transvecting (\ref{10.4}) by $b_i$, we get
\begin{eqnarray*}
-\frac{\beta r_{00}+\alpha^2 s_0}{b^2\alpha^2}\beta-\frac{\alpha^2 s_0}{2\beta}+\frac{\beta r_{00}+\alpha^2 s_0}{2\beta}=P\beta,
\end{eqnarray*}
and substituting (\ref{10.5}) in the last equation, we obtain
\begin{eqnarray*}
-\frac{\beta r_{00}+\alpha^2 s_0}{b^2\alpha^2}\beta-\frac{\alpha^2 s_0}{2\beta}+\frac{\beta r_{00}+\alpha^2 s_0}{2\beta}=-\frac{\beta r_{00}+\alpha^2 s_0}{2b^2\alpha^2}\beta,
\end{eqnarray*}
that is,
\begin{eqnarray}\label{10.6}
\alpha^2(b^2 r_{00}-  s_0\beta) =\beta^2 r_{00}.
\end{eqnarray}
Since $\beta^2$ is not divisible by $\alpha^2$, it follows that $r_{00}$ must be divisible by $\alpha^2$, that is, there exists a function $c(x)$ of $x$ alone such that
\begin{eqnarray}\label{10.7}
  r_{00}=c(x)\alpha^2.
\end{eqnarray}
Substituting (\ref{10.7}) to (\ref{10.6}), we have
\begin{eqnarray}\label{10.8}
c(x)b^2\alpha^2-  s_0\beta =c(x)\beta^2 .
\end{eqnarray}
Since $c(x)b^2\alpha^2$ must be divisible by $\beta$ and  $\alpha^2$ is not divisible by $\beta$, it follows that $c(x)=0$.
Substituting $c(x)=0$ to (\ref{10.7}) and (\ref{10.8}), we have
\begin{eqnarray}\label{10.9}
   r_{ij}=0, \hspace{0.1in}  s_i=0.
\end{eqnarray}
Then from (\ref{10.5}) it follows $P=0$.
Lastly, substituting $P=0$ and (\ref{10.9}) to  (\ref{10.4}), we get ${s^i}_0=0$, that is,
\begin{eqnarray}\label{10.10}
s_{ij}=0.
\end{eqnarray}

From (\ref{10.9}) and (\ref{10.10}), we get
\begin{eqnarray*}
b_{i;j}=0,
\end{eqnarray*}
that is, $b_i$ is parallel with respect to $\alpha$.

Conversely, suppose that $b_i$ is parallel with respect to $\alpha$, then we have $B^i=0$,
that is Kropina metric is projectively equivalent to Riemannian metric $\alpha$.
Furthermore, we have $G^i={{\gamma_0}^i}_0$.

Summarizing the above discussion, we obtain
\begin{proposition}
The necessary and sufficient condition for a Kropina metric $F=\frac{\alpha^2}{\beta}$ to be projectively equivalent to the Riemannian metric $\alpha$ is that the vector field $(b_i)$ is parallel with respect to $\alpha$.
\end{proposition}

\subsection{The  condition  $b_{i;j}=0$ in terms of the navigation   data $(h, W)$ }

By a straightforward computation we have
\begin{eqnarray}\label{10.11}
	r_{ij}=2e^{-\kappa} \bigg(\texttt{R}_{ij}-\frac{1}{2}W^r\kappa_rh_{ij}\bigg),\hspace{0.1in}
	s_{ij}=2e^{-\kappa} \bigg(\texttt{S}_{ij}+\frac{\kappa_iW_j-\kappa_jW_i}{2}\bigg),
\end{eqnarray}
where $W^r=h^{rj}W_j$, and $k_i=\frac{\partial k}{\partial x^i}$, and we  put
\begin{eqnarray*}
	\texttt{R}_{ij}:=\frac{W_{i|j}+W_{j|i}}{2}, \hspace{0.1in}
	\texttt{S}_{ij}:=\frac{W_{i|j}-W_{j|i}}{2},
\end{eqnarray*}
where the notation "$_{|}$" stands for the covariant derivative with respect to $h$.

Suppose that the equation  $b_{i;j}=0$ holds, then 
\begin{eqnarray}\label{10.12}
	r_{ij}=	s_{ij}=0.
\end{eqnarray}
Furthermore, from the equation  $b_{i;j}=0$ it follows that $b^2$ is constant, hence 
 the equation $b^2e^\kappa=4$ implies that  $\kappa$ is also constant.
Since $\kappa_r=0$, the equations (\ref{10.11}) reduce to
\begin{eqnarray*}
	\texttt{R}_{ij}=0,
	\texttt{S}_{ij}=0.
\end{eqnarray*}
Hence we get
\begin{eqnarray}\label{10.13}
	W_{i|j}=0.
\end{eqnarray}

Conversely, suppose that the equation  (\ref{10.13}) holds and that $\kappa$ is constant, then the equation (\ref{10.11}) reduces to (\ref{10.12}), and therefore
\begin{eqnarray*}
	b_{i;j}=0.
\end{eqnarray*}




\vspace{0in}
\begin{center}
            Sorin V. Sabau\\
     Department of Mathematics\\
     Tokai University\\
       Sapporo, 005\,--\,8601 Japan\\
        E-mail:    sorin@tokai.ac.jp
\end{center}

\begin{center}
	K. Shibuya\\
	Department of Mathematics\\
	Hiroshima University\\
	Higashi Hiroshima, 739 \,--\, 8521 Japan\\
	E-mail:  shibuya@hiroshima-u.ac.jp  
\end{center}

\begin{center}
      R. Yoshikawa   \\
      Hino  High School \\                 %
	150 Kouzukeda\\
     Hino-cho Gamou-gun 
      529-1642 Japan \\              %
	E-mail: ryozo@e-omi.ne.jp
  \end{center}

\end{document}